\documentclass[reqno]{amsart}
\usepackage{url}
\newtheorem{theorem}{Theorem}[section]
\newtheorem{lemma}[theorem]{Lemma}
\newtheorem{corollary}[theorem]{Corollary}

\theoremstyle{definition}
\newtheorem{definition}[theorem]{Definition}

\theoremstyle{remark}

\numberwithin{equation}{section}
\usepackage{algorithmic}
\usepackage[top=3cm,bottom=2cm,right=2cm,left=2cm]{geometry}

\begin{document}

\title[A Single-Variable Proof of the Omega SPT Congruence Family]{A Single-Variable Proof of the Omega SPT Congruence Family Over Powers of 5}


\author{}
\address{}
\curraddr{}
\email{}
\thanks{}

\author{Nicolas Allen Smoot}
\address{}
\curraddr{}
\email{}
\thanks{}



\date{}

\dedicatory{}

\begin{abstract}
In 2018 Liuquan Wang and Yifan Yang proved the existence of an infinite family of congruences for the smallest parts function corresponding to the third order mock theta function $\omega(q)$.  Their proof took the form of an induction requiring 20 initial relations, and utilized a space of modular functions isomorphic to a free rank 2 $\mathbb{Z}[X]$-module.  This proof strategy was originally developed by Paule and Radu to study families of congruences associated with modular curves of genus 1.  We show that Wang and Yang's family of congruences, which is associated with a genus 0 modular curve, can be proved using a single-variable approach, via a ring of modular functions isomorphic to a localization of $\mathbb{Z}[X]$.  To our knowledge, this is the first time that such an algebraic structure has been applied to the theory of partition congruences.  Our induction is more complicated, and relies on sequences of functions which exhibit a somewhat irregular 5-adic growth.  However, the proof ultimately rests upon the direct verification of only 10 initial relations, and is similar to the classical methods of Ramanujan and Watson.
\end{abstract}

\maketitle

\section{Introduction}

In 2018 Liuquan Wang and Yifan Yang proved \cite{Wang2} the existence of an infinite family of congruences for the smallest parts function for the mock theta function $\omega(q)$.  Wang and Yang's proof employed methods developed by Paule and Radu.  However, the author has discovered that this family of congruences lends itself to a single-variable proof similar in theory to the classical methods of Ramanujan and Watson.  In such a proof certain very interesting algebraic intricacies emerge, which---to our knowledge---are different from all known proofs of partition congruence families.  The purpose of this paper is to elaborate this alternative proof.

The congruence family in question was originally conjectured \cite{Wang} by Wang in 2017, and concerns the $\omega(q)$ mock theta function analogue to the partition $\mathrm{spt}$ function studied by Andrews \cite{Andrews2}.

\begin{theorem}\label{wyfirst}
Let $\lambda_{\alpha}\in\mathbb{Z}$ be the minimal positive solution to $12x\equiv 1\pmod{5^{\alpha}}.$  Then
\begin{align}
\mathrm{spt}_{\omega}\left(2\cdot 5^{\alpha}n+\lambda_{\alpha}\right)\equiv 0\pmod{5^{\alpha}}.
\end{align}
\end{theorem}  Wang and Yang prove this theorem \cite{Wang2} by relating $\mathrm{spt}_{\omega}$ to the spt functions for certain Bailey pairs $C1$, $C5$ studied by Garvan and Jennings--Shaffer \cite{Garvan2}, as well as the function $c$, defined by

\begin{align}
\sum_{n=0}^{\infty}c(n)q^n := \frac{2E_2(2\tau)-E_2(\tau)}{(q^2;q^2)_{\infty}},\label{E2}
\end{align} with $q:=e^{2\pi i\tau}$, $\tau\in\mathbb{H}$, and $E_2(\tau)$ defined as the normalized weight 2 Eisenstein series (disregarding the nonholomorphic term):

\begin{align*}
E_2(\tau) := 1 - 24\sum_{n=1}^{\infty}\frac{nq^n}{1-q^n}.
\end{align*}  Wang and Yang then show that Theorem \ref{wyfirst} is a consequence of the following:

\begin{theorem}\label{Thm12}
Let $12n\equiv 1\pmod{5^{\alpha}}$.  Then $c(n)\equiv 0\pmod{5^{\alpha}}$.
\end{theorem}

This family of congruences has a form which is common to the field; in addition to Ramanujan's iconic congruences, \cite{Ramanujan}, \cite{Watson}, \cite{Atkin}, \cite{Knopp} an enormous number of similar results have been found; recent results include \cite{Garvan}, \cite{Paule}, \cite{Smoot}.  The proof of each of these results generally involves an induction argument which takes advantage of some algebraic structure on the associated space of modular functions.

In the case of Ramanujan's classical result for powers of 5, one works over the congruence subgroup $\Gamma_0(5)$ and an associated space of modular functions under specific meromorphic conditions.  Because the algebraic structure of this space of functions is isomorphic to $\mathbb{Z}[X]$, the full family of congruences may be proved by a straightforward process.  This is similarly true for Ramanujan's modified congruence family for powers of 7.  The general technique was first published by Watson \cite{Watson}, although Ramanujan appears to have been the first to understand it \cite{BO}.

However, Ramanujan's congruence family for powers of 11 is more difficult to prove.  This is owed in large measure to the fact that the underlying modular curve $\mathrm{X}_0(11)$ associated with $\Gamma_0(11)$ has genus 1.  To compare, $\mathrm{X}_0(5)$ and $\mathrm{X}_0(7)$ each have genus 0.  This fact ensures that, as a consequence of the Weierstrass gap theorem \cite{Paule3}, the necessary space of functions over $\Gamma_0(11)$ is isomorphic not to $\mathbb{Z}[X]$, but rather to a rank 2 $\mathbb{Z}[X]$-module.

The necessary modification of Ramanujan and Watson's approach was first developed by Atkin \cite{Atkin}.  Later families of congruences associated with a genus 1 modular curve include the Andrews--Sellers conjecture, proved by Paule and Radu \cite{Paule}, and the Choi--Kim--Lovejoy conjecture, proved by the author \cite{Smoot}.  The method used for both of these proofs was developed by Paule and Radu as an important modification of Atkin's approach.

Wang and Yang also employ this method.  In particular, they develop a sequence of weakly holomorphic modular forms $\left(L_{\alpha}\right)_{\alpha\ge 1}$ such that

\begin{align*}
L_{\alpha} = \Phi_{\alpha}\cdot\sum_{n=0}^{\infty}c\left(5^{\alpha}n + \lambda_{\alpha}\right)q^{n+1},
\end{align*} with $\Phi_{\alpha}$ an integer power series with constant term 1, and $\lambda_{\alpha}$ the minimum positive solution to $12x\equiv 1\pmod{5^{\alpha}}$ (see Section 2).  They show that

\begin{align}
\frac{L_{\alpha}}{5^{\alpha}\cdot F} = f_{0,\alpha}(t) + \rho\cdot f_{1,\alpha}(t),\label{WY1}
\end{align} in which $f_{i,\alpha}\in\mathbb{Z}[X]$,

\begin{align*}
F &:= F(\tau) = \frac{1}{24} \left( 50 E_2(10\tau)-25 E_2(5\tau) - 2 E_2(2\tau) + E_2(\tau) \right)
\end{align*} is a weight 2 holomorphic modular form, and $t,\rho$ are modular functions which take the form of eta quotients with integer expansions in the Fourier variable $q$.  The relevant congruence subgroup is $\Gamma_0(10)$.

This is standard to Paule and Radu's approach.  Note the free rank 2 $\mathbb{Z}[X]$-module structure of (\ref{WY1}), which is characteristic of the method.

However, Paule and Radu developed their method in order to overcome the complications which arise from a congruence family in which the associated modular curve has \textit{nonzero genus}.  The genus of $\mathrm{X}_0(10)$ is 0.

It is this extremely important and telling fact that drove us to attempt a more classical proof of Wang and Yang's theorem.

As an example, we take the first case of Theorem \ref{Thm12}.  Define 

\begin{align}
L_{1} &= (q^{10};q^{10})_{\infty} \sum_{n=0}^{\infty}c\left(5n + 3\right)q^{n+1},\label{L1def}
\end{align} in a manner standard to the theory (see Section 2).  Wang and Yang prove that $L_1\equiv 0\pmod{5}$ by showing that

\begin{align}
L_1 = F\cdot \left( \left(245 t+3750 t^2+15625 t^3\right) - \rho\cdot\left( 125 t + 3125 t^2 \right) \right),\label{L1WY}
\end{align} with $t$ and $\rho$ defined as in (\ref{WY1}).  However, we were able to find a function $y$, also modular over $\Gamma_0(10)$, with the following:

\begin{align}
L_1 =& \frac{F}{(1 + 5 y)^3}\cdot\left( 120 y + 1805 y^2 + 12050 y^3 + 39500 y^4 + 50000 y^5\right).\label{L1S}
\end{align}  If we note that $F$ and $(1+5y)^{-1}$ both expand into integer power series in the Fourier variable $q$ with the constant term 1, then we need only examine the remaining portion of the expression---a single-variable polynomial in $y$---to determine divisibility.  Not only does divisibility by 5 emerge very naturally, but there is also a curious occurence of powers of 2 which is not as easily visible in the identity by Wang and Yang.

An interesting complication emerges in the factor $(1+5y)^{-3}$.  One might correctly guess that our relevant space of modular functions for all $\alpha\ge 1$ is isomorphic to a localization of $\mathbb{Z}[X]$, rather than to $\mathbb{Z}[X]$ itself.  Indeed, we have the following remarkable result, the principal theorem of this paper:

\begin{theorem}\label{Mythm}
Let

\begin{align*}
\psi(\alpha):=\left\lfloor \frac{5^{\alpha+1}}{12} \right\rfloor + 1.
\end{align*}  For all $\alpha\ge 1$,

\begin{align}
\frac{(1+5y)^{\psi(\alpha)}}{5^{\alpha}\cdot F}\cdot L_{\alpha}\in\mathbb{Z}[y].\label{myeqn}
\end{align}

\end{theorem}

We are not aware of any other congruence families in which the proof necessitates such a ring structure, and it would be interesting to know whether any additional examples exist.  Together with Silviu Radu, the author has developed some algorithmic machinery \cite{Raduns}, using the theory of modular functions, which might easily be modified to examine identities with a form similar to (\ref{L1S}).  The author strongly believes that localized rings may yet prove to be an enormously productive environment in which to examine new arithmetic properties in partition theory, especially for situations in which more traditional methods fail.

Another interesting difficulty arises in the somewhat irregular 5-adic growth of each term of $L_{\alpha}$ under repeated application of the corresponding $U_5$ operators.  In particular, in mapping $L_{2\alpha-1}$ to $L_{2\alpha}$, the individual components of the linear coefficient do not increase piecewise with respect to their 5-adic value---rather, the components must be shown \textit{to sum to} the necessary multiple of 5 (see Definition \ref{3termad} and Theorem \ref{w12w1}).

Such a strange complication necessitates a very precise manipulation on the 5-adic growth of our critical functions, together with a careful examination of the coefficients in our auxiliary functions modulo 5.  As in the matter of localization, we are unaware of any other examples of congruence families which demand such a constructive method of verifying divisibility by 5.

A final complication emerges in the base cases of our key lemmas.  We require the verification of 50 initial relations.  However, we can show that these 50 are algebraically dependent, and that \textit{a total of only 10 initial relations} need be directly established.  This stands in contrast to the 20 that Wang and Yang require for their proof.  From these 10 relations, the 50 relations necessary for our induction may be assembled and verified with relative ease through a computer algebra system.  The computational complexity is striking; nevertheless, the reliance (in principle, at least) upon so few relations, together with the single-variable approach, compels us to call our proof classical, or perhaps ``semiclassical.''

In total, these complications seem overwhelming, and it is understandable that a single-variable proof has not been found before now.  It seems that the genus of the underlying modular curve alone is sufficient to compel a single-variable proof, in spite of the many considerable difficulties.

The remainder of our paper is outlined as follows:  In Section 2 we define the necessary functions $L_{\alpha}$, along with certain auxiliary functions.  We also define the modular function $y$, from which we will develop certain modular equations.  In Section 3 we will develop some important module structures, including our polynomial localization.  Following this, we develop our modified $U_5$ operators, and prove certain key properties of $U_5$ on our localized ring, including some important arithmetical information.

In Section 4 we carefully demonstrate 5-adic convergence of our relevant space of functions.  In particular, we impose certain additional arithmetic constraints on our functions in order to control the somewhat irregular behavior of our 5-adic growth, especially over the coefficients of the smaller powers of $y$.  This culminates in the proof of our Main Theorem.  As a reference, we provide certain useful $5$-adic valuation tables in Appendix I.

In Section 5, we outline the proof of our fundamental relations, and the algebraic method by which we may verify the 50 initial relations needed for the induction of our Main Lemma.  The proof of the ten fundamental relations is based on the cusp analysis established from the theory of modular functions.  We give these ten fundamental relations in Appendix II.  However, for want of space, we will post the computations of our 50 initial relations (along with our cusp analysis calculations, and some miscellaneous calculations) in an online Mathematica supplement, which can be found at \url{https://www3.risc.jku.at/people/nsmoot/online3.nb}.

The importance of a computational approach to this problem cannot be overstated.  Not only was there a need to calculate certain relations to complete our induction (most of which would prove too tedious to demonstrate by hand); but many of our results which can in principle be proved without any reliance on computation (e.g., the sufficiency of the additional conditions imposed in Definition \ref{3termad}) were only found through many weeks of exhaustive experimentation.  

In our final section we attempt to provide some insight as to why our technique cannot be applied to a family of congruences in which the associated modular curve has genus 1, e.g., the Andrews--Sellers family.

\section{Key Functions}

Hereafter, we denote $q:=e^{2\pi i\tau}$, for $\tau\in\mathbb{H}$.  We begin by defining an important auxiliary function:

\begin{align}
Z &:= Z(\tau) = \frac{\eta(50\tau)}{\eta(2\tau)}.\label{Zdef}
\end{align}  Here $Z(\tau)$ is a modular function over $\Gamma_0(10)$ (See~Section~5).

We will now define our key generating functions, and their behavior under the standard $U_5$ operator.

\subsection{Generating Functions}

Our main generating functions $L_{\alpha}$ for each case of Theorem \ref{Thm12} are defined as follows:
\begin{align}
L_0 &:= 2E_2(2\tau)-E_2(\tau),\label{L0d}\\
L_{2\alpha-1} &:= (q^{10};q^{10})_{\infty} \sum_{n=0}^{\infty}c\left(5^{2\alpha-1}n + \lambda_{2\alpha-1}\right)q^{n+1},\label{Lod}\\
L_{2\alpha} &:=(q^{2};q^{2})_{\infty} \sum_{n=0}^{\infty}c\left(5^{2\alpha}n + \lambda_{2\alpha}\right)q^{n+1},\label{Led}
\end{align} with the $\lambda_{\alpha}$ defined as

\begin{align}
\lambda_{2\alpha-1} &:= \frac{1+7\cdot 5^{2\alpha-1}}{12},\label{lamodef}\\
\lambda_{2\alpha} &:= \frac{1+11\cdot 5^{2\alpha}}{12}\label{lamedef}.
\end{align}  In either case, $\lambda_{\alpha}\in\mathbb{Z}$ are the minimal positive solutions to 

\begin{align*}
12x\equiv 1\pmod{5^{\alpha}}.
\end{align*}  Therefore, one could write $L_{\alpha}$ in the form

\begin{align*}
L_{\alpha} = \Phi_{\alpha}(q)\cdot \sum_{12n\equiv 1\bmod{5^{\alpha}}} c(n)q^{\left\lfloor \frac{n}{5^{\alpha}} \right\rfloor}.
\end{align*}  We now define the standard Hecke $U_5$ operator:

\begin{definition}

Let $f(q) = \sum_{m\ge M}a(m)q^m$.  Then define

\begin{align}
U_5\left(f(q)\right) = \sum_{5m\ge M} a(5m)q^m.\label{defU5}
\end{align}

\end{definition}

We list some of the important properties of $U_5$.  The proofs are straightforward, and can be found in \cite[Chapter 10]{Andrews} and \cite[Chapter 8]{Knopp}.

\begin{lemma}\label{impUprop}

Given two functions 

\begin{align*}
f(q) = \sum_{m\ge M}a(m)q^m,\ g(q) = \sum_{m\ge N}b(m)q^m,
\end{align*} any $\alpha\in\mathbb{C}$, a primitive fifth root of unity $\zeta$, and the convention that $q^{1/5}~=~e^{2\pi i\tau/5}$, we have the following:
\begin{enumerate}
\item $U_5\left(\alpha\cdot f+g\right) = \alpha\cdot U_5\left(f\right) + U_5\left(g\right)$;
\item $U_5\left(f(q^5)g(q)\right) = f(q) U_5\left(g(q)\right)$;
\item $5\cdot U_5\left(f\right) = \sum_{r=0}^4 f\left( \zeta^rq^{1/5} \right)$.
\end{enumerate}
\end{lemma}

The $U_5$ operator provides us with a convenient means of accessing $L_{\alpha+1}$ from $L_{\alpha}$, as the following lemma shows:

\begin{lemma}
For all $\alpha\ge 0$, we have
\begin{align}
L_{2\alpha} &= U_5\left( L_{2\alpha-1} \right),\label{Lote}\\
L_{2\alpha+1} &= U_5\left( Z\cdot L_{2\alpha} \right).\label{Leto}
\end{align}
\end{lemma}

\begin{proof}
For any $\alpha\ge 1$,

\begin{align*}
U_5\left( L_{2\alpha-1} \right) &= U_5\left( (q^{10};q^{10})_{\infty} \sum_{n\ge 0} c\left(5^{2\alpha-1}n + \lambda_{2\alpha-1}\right)q^{n+1} \right)\\
&= (q^{2};q^{2})_{\infty}\cdot U_5\left( \sum_{n\ge 1}c\left(5^{2\alpha-1}(n-1) + \lambda_{2\alpha-1}\right)q^{n} \right)\\
&= (q^{2};q^{2})_{\infty}\cdot \sum_{5n\ge 1}c\left(5^{2\alpha-1}(5n-1) + \lambda_{2\alpha-1}\right)q^{n}\\
&= (q^{2};q^{2})_{\infty}\cdot \sum_{n\ge 1}c\left(5^{2\alpha}n-5^{2\alpha-1} + \lambda_{2\alpha-1}\right)q^{n}\\
&= (q^{2};q^{2})_{\infty}\cdot \sum_{n\ge 0}c\left(5^{2\alpha}n+5^{2\alpha}-5^{2\alpha-1} + \lambda_{2\alpha-1}\right)q^{n+1}\\
&= (q^{2};q^{2})_{\infty}\cdot \sum_{n\ge 0}c\left(5^{2\alpha}n+\lambda_{2\alpha}\right)q^{n+1}.
\end{align*}  Similarly,

\begin{align*}
U_5\left( Z\cdot L_{2\alpha} \right) &= U_5\left( q^2\frac{(q^{50};q^{50})_{\infty}}{(q^{2};q^{2})_{\infty}}(q^{2};q^{2})_{\infty} \sum_{n\ge 0} c\left(5^{2\alpha}n + \lambda_{2\alpha}\right)q^{n+1} \right)\\
&= (q^{10};q^{10})_{\infty}\cdot U_5\left( \sum_{n\ge 3}c\left(5^{2\alpha}(n-3) + \lambda_{2\alpha}\right)q^{n} \right)\\
&= (q^{10};q^{10})_{\infty}\cdot \sum_{5n\ge 3}c\left(5^{2\alpha}(5n-3) + \lambda_{2\alpha}\right)q^{n}\\
&= (q^{10};q^{10})_{\infty}\cdot \sum_{n\ge 1}c\left(5^{2\alpha+1}n-3\cdot 5^{2\alpha} + \lambda_{2\alpha}\right)q^{n}\\
&= (q^{10};q^{10})_{\infty}\cdot \sum_{n\ge 0}c\left(5^{2\alpha+1}(n+1)-3\cdot 5^{2\alpha} + \lambda_{2\alpha}\right)q^{n+1}\\
&= (q^{10};q^{10})_{\infty}\cdot \sum_{n\ge 0}c\left(5^{2\alpha+1}n+\lambda_{2\alpha+1}\right)q^{n+1}.
\end{align*}
\end{proof}

\subsection{The Modular Equations}

Our most important functions are the following:

\begin{align}
x = x(\tau) :=& \prod_{m=1}^{\infty}\frac{(1-q^{2m})^5(1-q^{5m})}{(1-q^m)^5(1-q^{10m})},\label{xdef}\\
y = y(\tau) :=& q\prod_{m=1}^{\infty}\frac{(1-q^{2m})(1-q^{10m})^3}{(1-q^m)^3(1-q^{5m})}\label{ydef}.
\end{align}  

Notice that, by the Freshman's Dream,

\begin{align*}
(1-q^m)^5&\equiv 1-q^{5m}\pmod{5},\\
(1-q^{2m})^5&\equiv 1-q^{10m}\pmod{5}.
\end{align*}  This yields

\begin{align}
& \prod_{m=1}^{\infty}\frac{(1-q^{2m})^5(1-q^{5m})}{(1-q^m)^5(1-q^{10m})}\equiv 1\pmod{5}.\label{xtoy1}
\end{align}  It is not difficult to verify that

\begin{align}
\frac{x-1}{5} = q\prod_{m=1}^{\infty}\frac{(1-q^{2m})(1-q^{10m})^3}{(1-q^m)^3(1-q^{5m})},\label{xtoy2}
\end{align} from which $x=1+5y$ follows.

\begin{theorem}
Define
\begin{align*}
a_0(\tau) &= -y - 5\cdot 4\cdot y^2 - 5^2\cdot 6\cdot y^3 - 5^3\cdot 4\cdot y^4 - 5^4\cdot y^5\\
a_1(\tau) &= -5\cdot 3 y - 5\cdot 61\cdot y^2 - 5^2\cdot 93\cdot y^3 - 5^3\cdot 63\cdot y^4 - 5^4\cdot 16\cdot y^5\\
a_2(\tau) &= -5\cdot 17\cdot y - 5^3\cdot 14\cdot y^2 - 5^2\cdot 541\cdot y^3 - 5^3\cdot 372\cdot y^4 - 5^4\cdot 96\cdot y^5\\
a_3(\tau) &= -5\cdot 43\cdot y - 5^2\cdot 179\cdot y^2 - 5^4\cdot 56\cdot y^3 - 5^3\cdot 976\cdot y^4 - 5^4\cdot 256\cdot y^5\\
a_4(\tau) &= -5\cdot 41\cdot y - 5^2\cdot 172\cdot y^2 - 5^3\cdot 272\cdot y^3 - 5^4\cdot 192 y^4 - 5^4\cdot 256\cdot y^5.
\end{align*}  Then we have

\begin{align}
y^5+\sum_{j=0}^4 a_j(5\tau) y^j = 0.\label{modY}
\end{align}
\end{theorem}

\begin{proof}
Because $y(5\tau)$ is a modular function with only one pole, we may prove this equation using cusp analysis.  See the end of Section 5.
\end{proof}

\begin{theorem}
Define
\begin{align*}
b_0(\tau) &=-x^5\\
b_1(\tau) &=1 + 5 x + 5 x^2 + 5 x^3 + 5 x^4 - 16 x^5\\
b_2(\tau) &=-4 - 5\cdot 3\cdot x + 5\cdot 2\cdot x^2 + 5\cdot 7\cdot x^3 + 5\cdot 12\cdot x^4 - 96 x^5\\
b_3(\tau) &=6 + 5\cdot 3\cdot x - 5\cdot 7 x^2 + 5\cdot 8 x^3 + 5\cdot 48\cdot x^4 - 256 x^5\\
b_4(\tau) &=-4 - 5 x + 5\cdot 4\cdot x^2 - 5\cdot 16\cdot x^3 + 5\cdot 64\cdot x^4 - 256 x^5.
\end{align*}  Then we have

\begin{align}
x^5+\sum_{k=0}^4 b_k(5\tau) x^k = 0.\label{modX}
\end{align}
\end{theorem}

\begin{proof}
Simply substitute $y=(x-1)/5$ into (\ref{modY}), and simplify.
\end{proof}  For convenience of notation, in later sections we will define $b_5(\tau) := 1$.

\section{Algebra Structure}

\subsection{Localized Ring}
We will begin to construct the algebra structure needed for our proof, beginning with the peculiar localization property.  Define the multiplicatively closed set
\begin{align}
\mathcal{S} := \left\{ (1+5y)^n : n\in\mathbb{Z}_{n\ge 0} \right\}.\label{Sdef}
\end{align}

We will prove that for every $\alpha\ge 1$, $L_{\alpha}$ is a member of the localization of $\mathbb{Z}[y]$ at ${\mathcal{S}}$, which we will denote by $\mathbb{Z}[y]_{\mathcal{S}}$.  Notice that because $1/x^n = 1/(1+5y)^n$ is an eta quotient with an integer power series expansion in $q=e^{2\pi i\tau}$ for every $n\ge 1$, we can expand every element of the localization into an integer power series in $q$, i.e., $\mathbb{Z}[y]_{\mathcal{S}}\subseteq\mathbb{Z}[[q]]$.

We need to define two general classes of subsets of $\mathbb{Z}[y]_{\mathcal{S}}$.  But first, we need a key definition:

\begin{definition}
Let $n\ge 1$.  A function $s:\mathbb{Z}\rightarrow\mathbb{Z}$ is discrete if $s(m)=0$ for sufficiently large $m$.  A function $h:\mathbb{Z}^{n}\rightarrow\mathbb{Z}$ is a discrete array if for any fixed $(m_1,m_2,...,m_{n-1})\in\mathbb{Z}^{n-1}$, $h(m_1,m_2,...,m_{n-1},m)$ is discrete with respect to $m$.
\end{definition}

We now need to construct suitable sets to contain our $L_{\alpha}$.  Due to a somewhat irregular pattern of 5-adic growth, we must define our 5-adic valuation function very carefully.

\begin{align*}
\theta(m) := \begin{cases}
\left\lfloor \frac{5m-5}{6} \right\rfloor,\ &1\le m\le 2,\\
\left\lfloor \frac{5m-5}{6} \right\rfloor - 1,\ &m\ge 3,
\end{cases}
\end{align*}

\begin{align*}
\phi(m) := \begin{cases}
\left\lfloor \frac{5m-5}{6} \right\rfloor,\ &1\le m\le 3,\\
\left\lfloor \frac{5m-5}{6} \right\rfloor - 1,\ &m\ge 4.
\end{cases}
\end{align*}

Now we take an arbitrary $n\ge 1$, and define the following:

\begin{align}
\mathcal{V}_n^{(0)}:=& \left\{ \frac{1}{(1+5y)^n}\sum_{m\ge 1} s(m)\cdot 5^{\theta(m)}\cdot y^m : s \text{ is discrete} \right\},\label{V0def}\\
\mathcal{V}_n^{(1)}:=& \left\{ \frac{1}{(1+5y)^n}\sum_{m\ge 1} s(m)\cdot 5^{\phi(m)}\cdot y^m : s \text{ is discrete} \right\}.\label{V1def}
\end{align}

\subsection{Recurrence Relation}

We now define the following maps,
\begin{align}
U^{(1-i)}\left( f \right) := \frac{U_5\left( F\cdot Z^i \cdot f \right)}{F}.\label{Uidef}
\end{align} for $i=0,1$.

Now we are ready to utilize our modular equations, together with our $U^{(i)}$ operators to build certain helpful recurrence relations.

\begin{lemma}
For all $m,n\in\mathbb{Z}$, and $i\in\{0,1\}$, we have
\begin{align}
U^{(i)}\left( \frac{y^m}{(1+5y)^n} \right) = -\frac{1}{(1+5y)^5}\sum_{j=0}^{4}\sum_{k=1}^{5} a_j(\tau)b_k(\tau)\cdot U^{(i)}\left( \frac{y^{m+j-5}}{(1+5y)^{n-k}} \right).\label{UmodX}
\end{align}
\end{lemma}

\begin{proof}
We can write 

\begin{align}
b_0(5\tau) &= -\sum_{k=1}^5 b_k(5\tau) x^k,\nonumber \\
1 &=-\frac{1}{b_0(5\tau)} \sum_{k=1}^5 b_k(5\tau) x^k,\nonumber\\
x^{-n} &=-\frac{1}{b_0(5\tau)} \sum_{k=1}^5 b_k(5\tau) x^{-(n-k)},\label{xnegn1}
\end{align} for $n\ge 1$.  Writing $x$ in terms of $y$, we have

\begin{align}
(1+5y)^{-n} &=-\frac{1}{b_0(5\tau)} \sum_{k=1}^5 b_k(5\tau) (1+5y)^{-(n-k)}.\label{xnegn2}
\end{align}  If we multiply both sides by $y^m$ for some $m\ge 1$, then

\begin{align}
\frac{y^m}{(1+5y)^n} &= -\frac{1}{b_0(5\tau)}\sum_{k=1}^{5} b_k(5\tau)\cdot \frac{y^{m}}{(1+5y)^{n-k}}\nonumber\\
&= -\frac{1}{(1+5y(5\tau))^5}\sum_{k=1}^{5} b_k(5\tau)\cdot \frac{y^{m}}{(1+5y)^{n-k}}.\label{xnegn3}
\end{align}  Now we expand each power of $y$ with its modular equation, and rearrange:

\begin{align}
\frac{y^m}{(1+5y)^n} &= -\frac{1}{b_0(5\tau)}\sum_{k=1}^{5} b_k(5\tau)\cdot\sum_{j=0}^4 a_j(5\tau) \frac{y^{m+j-5}}{(1+5y)^{n-k}}\nonumber\\
&= -\frac{1}{(1+5y(5\tau))^5}\sum_{j=0}^4\sum_{k=1}^{5}a_j(5\tau) b_k(5\tau)\cdot \frac{y^{m+j-5}}{(1+5y)^{n-k}}.\label{xnegn4}
\end{align}  Now multiply both sides by $F\cdot Z^{1-i}$:

\begin{align}
F\cdot Z^{1-i}\cdot\frac{y^m}{(1+5y)^n} &= -\frac{1}{b_0(5\tau)}\sum_{k=1}^{5} b_k(5\tau)\cdot\sum_{j=0}^4 a_j(5\tau) \frac{y^{m+j-5}}{(1+5y)^{n-k}}\nonumber\\
= -\frac{1}{(1+5y(5\tau))^5}&\sum_{j=0}^4\sum_{k=1}^{5}a_j(5\tau) b_k(5\tau)\cdot F\cdot Z^{1-i}\cdot \frac{y^{m+j-5}}{(1+5y)^{n-k}}.\label{xnegn5}
\end{align}

We are now ready to take the $U_5$ operator.  Recall that by Line 2 of Lemma \ref{impUprop}, for any functions $f(\tau),g(\tau)$,

\begin{align*}
U_5(f(5\tau)\cdot g(\tau)) = f(\tau)\cdot U_5(g(\tau)).
\end{align*} This gives us

\begin{align}
U_5&\left(F\cdot Z^{1-i}\cdot \frac{y^m}{(1+5y)^n} \right)\nonumber\\ =& -\frac{1}{(1+5y)^5}\sum_{j=0}^{4}\sum_{k=1}^{5} a_j(\tau)b_k(\tau)\cdot U_5\left(F\cdot Z^{1-i}\cdot \frac{y^{m+j-5}}{(1+5y)^{n-k}} \right).\label{xnegn6}
\end{align}  Dividing both sides by $F$, we achieve our formula.

\end{proof}

\subsection{Main Lemma}

We need to provide certain general relations for $U^{(i)}\left( \frac{y^m}{(1+5y)^n} \right)$.  For this we will define the following:

\begin{align*}
\pi_1(m,r) &:= \begin{cases}
0,\ &1\le m\le 2, \text{ and } r=1,\\
3,\ &1\le m\le 2, \text{ and } r=3,\\
\left\lfloor \frac{5r+1}{6} \right\rfloor,\ &1\le m\le 2, \text{ and } r\ge 2, \text{ and } r\neq 3,\\
2,\ &m=3, \text{ and } r=2,\\
\left\lfloor \frac{5r-2}{6} \right\rfloor,\ &m=3, \text{ and } r\neq 2,\\
\left\lfloor \frac{5r-m+1}{6} \right\rfloor,\ &m\ge 4;
\end{cases}\\
\pi_0(m,r) &:= \begin{cases}
\left\lfloor \frac{5r+1}{6} \right\rfloor,\ &m=1,\\
\left\lfloor \frac{5r+1}{6} \right\rfloor,\ &m = 2, \text{ and } r\neq 3, 4, 5,\\
\left\lfloor \frac{5r-5}{6} \right\rfloor,\ &m = 2, \text{ and } 3\le r \le 5,\\
\left\lfloor \frac{5r-m-2}{6} \right\rfloor,\ &m\ge 3,
\end{cases}
\end{align*}\\

\begin{theorem}\label{thmuyox}
There exist discrete arrays $h_1, h_0 : \mathbb{Z}^{3}\rightarrow\mathbb{Z}$ such that\\
\begin{align}
U^{(1)}\left( \frac{y^m}{(1+5y)^n} \right) &= \frac{1}{(1+5 y)^{5n-4}}\sum_{r\ge \left\lceil m/5\right\rceil} h_1(m,n,r)\cdot 5^{\pi_1(m,r)}\cdot y^r\label{U1YoX},\\
U^{(0)}\left( \frac{y^m}{(1+5y)^n} \right) &= \frac{1}{(1+5 y)^{5n-2}}\sum_{r\ge \left\lceil (m+2)/5\right\rceil} h_0(m,n,r)\cdot 5^{\pi_0(m,r)}\cdot y^r\label{U0YoX}.
\end{align}
\end{theorem}

Notice that

\begin{align}
\pi_1(m,r)\ge&\left\lfloor \frac{5r-m+1}{6} \right\rfloor,\label{pimd1}\\
\pi_0(m,r)\ge&\left\lfloor \frac{5r-m-2}{6} \right\rfloor\label{pimd0}.
\end{align}  We will therefore begin by proving the following:

\begin{lemma}\label{lemmaMN}
Let $\kappa,\delta,\mu\in\mathbb{Z}_{\ge 0}$ be fixed, and fix $i$ to either 0 or 1.  If there exists a discrete array $h_i$ such that
\begin{align}
U^{(i)}\left( \frac{y^{m}}{(1+5y)^{n}} \right) &= \frac{1}{(1+5y)^{5n-\kappa}}\sum_{r\ge \left\lceil \frac{m+\delta}{5} \right\rceil} h_i(m,n,r)\cdot 5^{\left\lfloor \frac{5r-m+\mu}{6} \right\rfloor}\cdot y^{r}
\end{align} for $1\le m\le 5$, $1\le n\le 5$, then such a relation can be made to hold for all $m\ge 1$, $n\ge 1$.
\end{lemma}

\begin{lemma}\label{lemmaN}
Let $\kappa,\delta\in\mathbb{Z}_{\ge 0}$ and $m_0\in\mathbb{Z}_{\ge 1}$ be fixed, and fix $i$ to either 0 or 1.  If there exists a discrete array $h_i$ such that
\begin{align}
U^{(i)}\left( \frac{y^{m_0}}{(1+5y)^{n}} \right) &= \frac{1}{(1+5y)^{5n-\kappa}}\sum_{r\ge \left\lceil \frac{m+\delta}{5} \right\rceil} h_i(m,n,r)\cdot 5^{\pi_i(m_0,r)}\cdot y^{r}
\end{align} for $1\le n\le 5$, then such a relation can be made to hold for all $n\ge 1$.
\end{lemma}

We may verify Theorem Lemma \ref{lemmaN} for $1\le m_0\le 4$, since for any larger $m$, (\ref{pimd1}), (\ref{pimd0}) are equalities.  If Lemma \ref{lemmaMN} is also satisfied, then Theorem \ref{thmuyox} follows.

\begin{proof}[Proof of Lemma \ref{lemmaMN}]

We will use induction.  Suppose that the relation holds for all positive integers strictly less than some $m_0,n_0\in\mathbb{Z}_{\ge 6}$.  We want to show that the relation can be made to hold for $m_0$ and $n_0$.  We have
\begin{align}
U^{(i)}\left( \frac{y^{m_0}}{(1+5y)^{n_0}} \right)&\nonumber\\ = -\frac{1}{(1+5y)^5}&\sum_{j=0}^{4}\sum_{k=1}^{5} a_j(\tau)b_k(\tau)\cdot U^{(i)}\left( \frac{y^{m_0+j-5}}{(1+5y)^{n_0-k}} \right)\\
= -\frac{1}{(1+5y)^5}\sum_{j=0}^{4}&\sum_{k=1}^{5} \frac{a_j(\tau)b_k(\tau)}{(1+5y)^{5(n_0-k)-\kappa}}\nonumber\\ \times\sum_{r\ge \left\lceil (m_0+j-5 + \delta)/5 \right\rceil}& h_i(m_0+j-5,n_0-k,r)\cdot 5^{\left\lfloor \frac{5r-(m_0+j-5)+\mu}{6} \right\rfloor}\cdot y^r\\
=\frac{1}{(1+5y)^{5n_0-\kappa}}&\sum_{j=0}^{4}\sum_{k=1}^{5} w(j,k)\nonumber\\ \times\sum_{r\ge \left\lceil (m_0+j-5+\delta)/5 \right\rceil} &h_i(m_0+j-5,n_0-k,r)\cdot 5^{\left\lfloor \frac{5r-(m_0+j-5)+\mu}{6} \right\rfloor}\cdot y^r,
\end{align} with

\begin{align}
w(j,k) :=& -a_j(\tau)b_k(\tau)(1+5y)^{5(k-1)}\nonumber\\
=& \sum_{l=1}^{25} v(j,k,l)\cdot 5^{\left\lfloor \frac{5l+j}{6} \right\rfloor}\cdot y^l.
\end{align}  This can be demonstrated by a simple expansion of $a_j(\tau)b_k(\tau)(1+5y)^{5(k-1)}$.  Expanding $w(j,k)$, we have

\begin{align}
&U^{(i)}\left( \frac{y^{m_0}}{(1+5y)^{n_0}} \right)\nonumber\\ =& \frac{1}{(1+5y)^{5n_0-\kappa}}\sum_{j=0}^{4}\sum_{k=1}^{5}\sum_{l=1}^{25} \sum_{r\ge \left\lceil (m_0+j-5+\delta)/5 \right\rceil}\nonumber\\ &v(j,k)\cdot h_i(m_0+j-5,n_0-k,r)\cdot 5^{\left\lfloor \frac{5r-(m_0+j-5)+\mu}{6} \right\rfloor + \left\lfloor \frac{5l+j}{6} \right\rfloor}\cdot y^{r+l}.\label{mnmodeqnrpl}
\end{align}  Notice that for any $M,N\in\mathbb{Z}$,

\begin{align*}
\left\lfloor \frac{M}{6} \right\rfloor + \left\lfloor \frac{N}{6} \right\rfloor \ge \left\lfloor \frac{M+N-5}{6} \right\rfloor.
\end{align*}  Because of this,

\begin{align}
\left\lfloor \frac{5r-(m_0+j-5)+\mu}{6} \right\rfloor + \left\lfloor \frac{5l+j}{6} \right\rfloor \ge \left\lfloor \frac{5(r+l) - m_0 + \mu}{6} \right\rfloor.
\end{align} And because

\begin{align}
r+l &\ge\left\lceil \frac{m_0+j-5+\delta}{5} \right\rceil + l\\ &\ge \left\lceil \frac{m_0+\delta}{5}-\frac{5-j}{5} \right\rceil + l\\
&\ge \left\lceil \frac{m_0+\delta}{5} \right\rceil - 1 + l\\ &\ge \left\lceil \frac{m_0+\delta}{5} \right\rceil,
\end{align} we can relabel our powers of $y$ so that

\begin{align}
&U^{(i)}\left( \frac{y^{m_0}}{(1+5y)^{n_0}} \right)\nonumber\\ =& \frac{1}{(1+5y)^{5n_0-\kappa}}\sum_{\substack{0\le j\le 4,\\ 1\le k\le 5,\\ 1\le l\le 25}} \sum_{r\ge \left\lceil \frac{m_0+\delta}{5} \right\rceil}\nonumber\\ &v(j,k)\cdot h_i(m_0+j-5,n_0-k,r-l)\cdot 5^{\left\lfloor \frac{5r-(m_0+j-5)+\mu}{6} \right\rfloor + \left\lfloor \frac{5l+j}{6} \right\rfloor}\cdot y^{r}.
\end{align}  If we define the discrete array $H_i$ by

\begin{align}
H_i(m,n,r) := \begin{cases}
\sum_{\substack{0\le j\le 4,\\ 1\le k\le 5,\\ 1\le l\le 25}} &\sum_{r\ge \left\lceil \frac{m+\delta}{5} \right\rceil - 1 + l} \hat{H}(i,j,k,l,r),\ r\ge l\\
0,\ r<l
\end{cases}\label{Hi}
\end{align} with

\begin{align*}
\hat{H}(i,j,k,l,r) := v(j,k)\cdot h_i(m+j-5,n-k,r-l)\cdot 5^{\epsilon(i,j,l,m,r)},
\end{align*}

\begin{align*}
\epsilon(i,j,l,m,r) := \left\lfloor \frac{5(r-l)-(m+j-5)+\mu}{6} \right\rfloor + \left\lfloor \frac{5l+j}{6} \right\rfloor -\left\lfloor \frac{5r - m_0 + \mu}{6} \right\rfloor,
\end{align*} then

\begin{align}
U^{(i)}\left( \frac{y^{m_0}}{(1+5y)^{n_0}} \right) &= \frac{1}{(1+5y)^{5n_0-\kappa}}\nonumber\\ &\times\sum_{r\ge \left\lceil \frac{m_0+\delta}{5} \right\rceil} H_i(m_0,n_0,r)\cdot 5^{\left\lfloor \frac{5r-m_0+\mu}{6} \right\rfloor}\cdot y^{r}.
\end{align}

\end{proof}

\begin{proof}[Proof of Lemma \ref{lemmaN}]

\begin{align}
U^{(i)}&\left( \frac{y^{m_0}}{(1+5y)^{n}} \right)\nonumber\\ =& -\frac{1}{(1+5y)^5}\sum_{k=1}^{5}b_k(\tau)\cdot U^{(i)}\left( \frac{y^{m_0}}{(1+5y)^{n-k}} \right)\\
=& -\frac{1}{(1+5y)^5}\sum_{k=1}^{5} \frac{b_k(\tau)}{(1+5y)^{5(n-k)-\kappa}} \sum_{r\ge 1} h_i(m_0,n-k,r)\cdot 5^{\pi_i(m_0,r)}\cdot y^r\\
=&\frac{1}{(1+5y)^{5n-\kappa}}\sum_{k=1}^{5} \hat{w}(k)\sum_{r\ge 1 } h_i(m_0,n-k,r)\cdot 5^{\pi_i(m_0,r)}\cdot y^r,
\end{align} with

\begin{align}
\hat{w}(k) :&= -b_k(\tau)(1+5y)^{5(k-1)}\nonumber\\
&=\begin{cases}
& \displaystyle{\sum_{l=0}^{20} \hat{v}(k,l)\cdot 5^{\left\lfloor \frac{5l+10}{6} \right\rfloor}\cdot y^l},\ k<5\\
&\displaystyle{1+\sum_{l=1}^{20} \hat{v}(5,l)\cdot 5^{\left\lfloor \frac{5l+10}{6} \right\rfloor}\cdot y^l},\ k=5.
\end{cases}
\end{align}  This can be demonstrated with a simple expansion of $\hat{w}(k)$.  Expanding, we have

\begin{align}
U^{(i)}&\left( \frac{y^{m_0}}{(1+5y)^{n}} \right) = \frac{1}{(1+5y)^{5n-\kappa}}\nonumber\\
\times&\Bigg(\sum_{\substack{1\le k\le 4,\\ 0\le l\le 20,\\ r\ge \left\lceil \frac{m_0+\delta}{5} \right\rceil}} \hat{v}(k,l)\cdot h_i(m_0,n-k,r)\cdot 5^{\pi_i(m_0,r) + \left\lfloor \frac{5l+10}{6} \right\rfloor}\cdot y^{r+l}\label{reliancenm5a3}\\
&+\sum_{\substack{1\le l\le 20,\\ r\ge \left\lceil \frac{m_0+\delta}{5} \right\rceil}} \hat{v}(5,l)\cdot h_i(m_0,n-5,r)\cdot 5^{\pi_i(m_0,r) + \left\lfloor \frac{5l+10}{6} \right\rfloor}\cdot y^{r+l}\label{reliancenm5a2}\\
&+\sum_{r\ge \left\lceil \frac{m_0+\delta}{5} \right\rceil} h_i(m_0,n-5,r)\cdot 5^{\pi_i(m_0,r)}\cdot y^{r}\Bigg).\label{reliancenm5a}
\end{align}  With a change of index, we have

\begin{align*}
U^{(i)}&\left( \frac{y^{m_0}}{(1+5y)^{n}} \right) = \frac{1}{(1+5y)^{5n-\kappa}}\nonumber\\
\times&\Bigg(\sum_{\substack{1\le k\le 4,\\ 0\le l\le 20,\\ r\ge l+ \left\lceil \frac{m_0+\delta}{5} \right\rceil}} \hat{v}(k,l)\cdot h_i(m_0,n-k,r-l)\cdot 5^{\pi_i(m_0,r-l) + \left\lfloor \frac{5l+10}{6} \right\rfloor}\cdot y^{r}\\
&+\sum_{\substack{1\le l\le 20,\\ r\ge l+ \left\lceil \frac{m_0+\delta}{5} \right\rceil}}\hat{v}(5,l)\cdot h_i(m_0,n-5,r-l)\cdot 5^{\pi_i(m_0,r-l) + \left\lfloor \frac{5l+10}{6} \right\rfloor}\cdot y^{r}\\
&+\sum_{r\ge \left\lceil \frac{m_0+\delta}{5} \right\rceil} h_i(m_0,n-5,r)\cdot 5^{\pi_i(m_0,r)}\cdot y^{r}\Bigg).
\end{align*}  Now,

\begin{align}
\pi_i(m_0,r-l) + \left\lfloor \frac{5l+10}{6} \right\rfloor \ge \pi_i(m_0,r).\label{inexeq}
\end{align}  This ensures that the critical 5-adic valuation of the terms of $U^{(i)}\left( \frac{y^{m_0}}{(1+5y)^{n}} \right)$ derives precisely from the sum

\begin{align}
\frac{1}{(1+5y)^{5n-\kappa}}&\sum_{r\ge \left\lceil \frac{m_0+\delta}{5} \right\rceil} h_i(m_0,n-5,r)\cdot 5^{\pi_i(m_0,r)}\cdot y^{r}\nonumber\\ &= U^{(i)}\left( \frac{y^{m_0}}{(1+5y)^{n-5}} \right).\label{reliancenm5b}
\end{align}

Therefore, if our relation is established for $1\le n\le 5$, then it must be true for all $n\ge 6$ as well.

We may now rearrange our sum and define a new discrete array in a manner similar to (\ref{Hi}) to finish the proof.

\end{proof}

\begin{proof}[Proof of Theorem \ref{thmuyox}]

These relations arise as consequences of Lemmas \ref{lemmaMN}, \ref{lemmaN} above, provided that the cases for $1\le m\le 5,$ $1\le n\le 5$ are established.  The computations needed to verify these relations are given in Section 5.  See our Mathematica supplement for the detailed computation.

\end{proof}

As an additional consequence of Lemmas \ref{lemmaMN}, \ref{lemmaN}, we have the following important result on the behavior of the coefficients in these expansions:

\begin{corollary}\label{congcor1}
For all $n\in\mathbb{Z}_{\ge 1}$ we have:
\begin{align}
h_0(1,n,1)&\equiv 1\pmod{5},\label{hmod2}\\
h_0(2,5n-4,1)&\equiv 0\pmod{5}\label{hmod3},\\
h_0(3,n,1)&\equiv 1\pmod{5},\label{hmod4}\\
h_0(1,n,2)&\equiv 4\pmod{5},\label{hmod2a}\\
h_0(2,5n-4,2)&\equiv 4\pmod{5},\label{hmod3b}\\
h_0(3,n,2)&\equiv 4\pmod{5},\label{hmod4a}\\
h_0(2,5n-4,3)&\equiv 1\pmod{5}.\label{hmod3c}
\end{align}  For all $n\in\mathbb{Z}_{\ge 1}$ and $1\le m\le 3$ we have:
\begin{align}
h_1(m,n,1)&\equiv 1\pmod{5}.\label{hmod1}
\end{align} 
\end{corollary}

\begin{proof}

We will first prove (\ref{hmod2})--(\ref{hmod4}).  Let us reexamine (\ref{reliancenm5a3}), (\ref{reliancenm5a2}), (\ref{reliancenm5a}).  Notice that whenever (\ref{inexeq}) is strict, i.e.,

\begin{align}
\pi_i(m_0,r-l) + \left\lfloor \frac{5l+10}{6} \right\rfloor > \pi_i(m_0,r),\label{isexeq}
\end{align} we must have $h_i(m_0,n,r)\equiv h_i(m_0,n-5,r)\pmod{5}$.  We therefore need only establish that (\ref{isexeq}) is true in all relevant cases.  Thereafter, we can simply compute the relevant coefficients for five consecutive values of $n$.

We note that for (\ref{reliancenm5a2}), $r\ge 1$ and $l\ge 1$.  Because of this, $r+l\ge 2$, and (\ref{reliancenm5a2}) will contribute nothing to the linear coefficient.  On the other hand, for (\ref{reliancenm5a3}), the only possibility is for $l=0$ and $r=1$.  Because $\left\lfloor \frac{5(0)+10}{6} \right\rfloor > 1$, we easily get (\ref{isexeq}).

Therefore, we must have 

\begin{align*}
h_0(1,n,1)&\equiv h_0(1,n-5,1) \pmod{5},\\
h_0(2,n,1)&\equiv h_0(2,n-5,1) \pmod{5},\\
h_0(3,n,1)&\equiv h_0(3,n-5,1) \pmod{5}.
\end{align*}  With our Mathematica supplement \url{https://www3.risc.jku.at/people/nsmoot/online3.nb}, we find that \\ \noindent$h_0(1,n,1)~\equiv~h_0(3,n,1)~\equiv~1\pmod{5}$ for $1\le n\le 5$.  Therefore, (\ref{hmod2})--(\ref{hmod3}) must be true for all $n$.  On the other hand, $h_0(2,n,1)$ varies regularly by the residue class modulo 5, and $h_0(2,n,1)\equiv 0\pmod{5}$ for $n\equiv 1\pmod{5}$.

To prove (\ref{hmod2a}), (\ref{hmod3b}), (\ref{hmod4a}), we note that we may directly compute $\pi_0(m,r)$.  Notice that the only way for $r+l=2$ to be true is for $r=l=1$ or $r=2$ and $l=0$.  For the first case, we have

\begin{align}
\pi_0(1,2-1) + \left\lfloor \frac{5(1)+10}{6} \right\rfloor &= 1 + 2 = 3 > 1 = \pi_0(1,1),\\
\pi_0(2,2-1) + \left\lfloor \frac{5(1)+10}{6} \right\rfloor &= 1 + 2 = 3 > 1 = \pi_0(2,1),\\
\pi_0(3,2-1) + \left\lfloor \frac{5(1)+10}{6} \right\rfloor &= 1 + 2 = 3 > 0 = \pi_0(3,1).
\end{align}  Here, (\ref{inexeq}) is strict; for the second case, i.e., for $r=2$ and $l=0$, (\ref{inexeq}) follows immediately.  Once again, we need only examine each case for five consecutive values of $n$.

To prove (\ref{hmod3c}), we take into account that there are three different ways for $r+l=3$ to be true.  Either $r=1$ and $l=2$, or $r=2$ and $l=1$, or $r=3, l=0$.  We therefore have

\begin{align}
\pi_0(2,3-2) + \left\lfloor \frac{5(2)+10}{6} \right\rfloor &= 1 + 3 = 4 > 1 = \pi_0(2,1),\\
\pi_0(2,3-1) + \left\lfloor \frac{5(1)+10}{6} \right\rfloor &= 1 + 2 = 3 > 1 = \pi_0(2,2),
\end{align} and the inequality is again trivially true in the case that $l=0$.  Because the inequality holds in both cases, we can again simply examine each case for five consecutive values of $n$.

Finally, to prove (\ref{hmod1}), we first note that for $m$ fixed, we may use the same reasoning as was used to prove (\ref{hmod3})--(\ref{hmod4}).  To see how $h_1(m,n,1)\pmod{5}$ varies with $m$, let us reexamine (\ref{mnmodeqnrpl}).  Notice that for $m_0\ge 6$, $U^{(i)}\left( \frac{y^{m_0}}{(1+5y)^{n_0}} \right)$ only possesses contributions for $U^{(i)}\left( \frac{y^{r+l}}{(1+5y)^{n_0-5}} \right)$, in which $r\ge 1$ and $l\ge 1$.  In other words, for $m\ge 6$, no contribution to the coefficient of $U^{(i)}\left( \frac{y^{1}}{(1+5y)^{n}} \right)$ can be made.

As only five values of $m$ will contribute to the coefficient that we want, we therefore only need to check (\ref{hmod1}) for $1\le m\le 5$, and for $1\le n\le 5$.
\end{proof}

\section{Main Theorem}

With the necessary relations established for $U^{(i)}\left( \frac{y^{m}}{(1+5y)^{n}} \right)$, we can now work towards the main theorem.  We begin with the following theorem:

\begin{theorem}\label{v02v1}
\begin{align}
\text{For every $f\in\mathcal{V}_n^{(0)}$, }\ \frac{1}{5}\cdot U^{(0)}\left( f \right)\in\mathcal{V}_{5n-2}^{(1)}.
\end{align}
\end{theorem}

\begin{proof}

Let $f\in\mathcal{V}_n^{(0)}$.  Then we can express $f$ as
\begin{align*}
f = \frac{1}{(1+5y)^n}\sum_{m\ge 1} s(m)\cdot 5^{\phi(m)}\cdot y^m.
\end{align*}  We write

\begin{align}
U^{(0)}\left( f \right) &= \sum_{m\ge 1} s(m)\cdot 5^{\phi(m)}\cdot U^{(0)}\left( \frac{y^m}{(1+5y)^n}\right)\\
=\frac{1}{(1+5y)^{5n-2}}&\sum_{m\ge 1}\sum_{r\ge \left\lceil (m+2)/5 \right\rceil} s(m)\cdot h_0(m,n,r) 5^{\phi(m) + \pi_0(m,r)}\cdot y^r\\
=\frac{1}{(1+5y)^{5n-2}}&\sum_{r\ge 1}\sum_{m\ge 1} s(m)\cdot h_0(m,n,r) 5^{\phi(m) + \pi_0(m,r)}\cdot y^r.
\end{align}  We examine $\phi(m) + \pi_0(m,r)$:

For $m=1$:

\begin{align*}
\phi(1) + \pi_0(1,r) & = 0 + \left\lfloor \frac{5r+1}{6} \right\rfloor \ge \theta(r)+1.
\end{align*}

For $m=2$:
\begin{align*}
\phi(2) + \pi_0(2,r) &= \begin{cases}
0 + \left\lfloor \frac{5r+1}{6} \right\rfloor ,&\ 1\le r\le 2 \text{ or } r\ge 6,\\
0 + \left\lfloor \frac{5r-5}{6} \right\rfloor ,&\ 3\le r\le 5.
\end{cases}
\end{align*}  In both cases, $\phi(2) + \pi_0(2,r)\ge \theta(r)+1.$

For $m=3$:

\begin{align*}
\phi(3) + \pi_0(3,r) & = 1 + \left\lfloor \frac{5r-5}{6} \right\rfloor \ge \theta(r)+1.
\end{align*}

For $m = 4$:

\begin{align*}
\phi(4) + \pi_0(4,r) & = 1 + \left\lfloor \frac{5r-6}{6} \right\rfloor = \left\lfloor \frac{5r}{6} \right\rfloor \ge \theta(r)+1
\end{align*}(remembering that $m\ge 4$ cannot contribute to the coefficient of $U^{(0)}\left( \frac{y^{1}}{(1+5y)^{n}} \right)$ since $\left\lceil (4+2)/5 \right\rceil = 2$).

For $m\ge 5$:

\begin{align*}
\phi(m) + \pi_0(m,r) & = \left\lfloor \frac{5m-5}{6} \right\rfloor -1 + \left\lfloor \frac{5r-m-2}{6} \right\rfloor\\
& \ge \left\lfloor \frac{5r-m-12}{6} \right\rfloor -1\\
&\ge \left\lfloor \frac{5r-5}{6} \right\rfloor + \left\lfloor \frac{4m-7}{6} \right\rfloor - 1\\
&\ge \left\lfloor \frac{5r-5}{6} \right\rfloor + 2 - 1\\
&\ge \left\lfloor \frac{5r-5}{6} \right\rfloor + 1\\
&\ge \theta(r)+1.
\end{align*}  Notice that, in all cases,

\begin{align*}
\phi(m) + \pi_0(m,r) \ge \theta(r)+1,
\end{align*} so that $U^{(0)}\left( f \right)\in\mathcal{V}_n^{(1)}$, with at least one additional power of 5.

\end{proof}

\begin{theorem}
Let $f\in\mathcal{V}_n^{(1)}$ and denote
\begin{align}
U^{(1)}\left( f \right) = \sum_{r\ge 1} \tilde{s}(r) \frac{y^r}{(1+5y)^{5n-4}}.
\end{align}  Then

\begin{align}
\frac{1}{5}\left(U^{(1)}\left( f \right) - \tilde{s}(1) \frac{y}{(1+5y)^{5n-4}}\right) \in\mathcal{V}_{5n-4}^{(0)}.
\end{align}
\end{theorem}

\begin{proof}
Let $f\in\mathcal{V}_n^{(1)}$.  Then we can express $f$ as
\begin{align*}
f = \frac{1}{(1+5y)^n}\sum_{m\ge 1} s(m)\cdot 5^{\theta(m)}\cdot y^m.
\end{align*}  We write

\begin{align}
U^{(1)}\left( f \right) &= \sum_{m\ge 1} s(m)\cdot 5^{\theta(m)}\cdot U^{(1)}\left( \frac{y^m}{(1+5y)^n}\right)\\
=\frac{1}{(1+5y)^{5n-4}}&\sum_{m\ge 1}\sum_{r\ge \left\lceil m/5 \right\rceil} s(m)\cdot h_1(m,n,r) 5^{\theta(m) + \pi_1(m,r)}\cdot y^r\\
=\frac{1}{(1+5y)^{5n-4}}&\sum_{r\ge 1}\sum_{m\ge 1} s(m)\cdot h_1(m,n,r) 5^{\theta(m) + \pi_1(m,r)}\cdot y^r.
\end{align}  Let us denote the coefficient of $\frac{y^r}{(1+5y)^{5n-4}}$ by $\tilde{s}(r)$.  Now we examine the 5-adic valuation of each component

Beginning with $1\le m\le 2$,

\begin{align}
\theta(1)+\pi_1(1,r) &=\begin{cases}
0,&\ r=1\\
3,&\ r=3\\
\left\lfloor \frac{5r+1}{6} \right\rfloor,&\ r=2 \text{ or } r\ge 4.
\end{cases}
\end{align}  With $m=3$,

\begin{align}
\theta(3)+\pi_1(3,r) &=\begin{cases}
0,&\ r=1\\
2,&\ r=2\\
\left\lfloor \frac{5r-2}{6} \right\rfloor,&\ r=3 \text{ or } r\ge 4.
\end{cases}
\end{align}  Notice that for $1\le m\le 3$, $\theta(m)+\pi_1(m,r) \ge \phi(r)+1$ \textit{except when} $r=1$.

Finally, for $m\ge 4$,

\begin{align}
\theta(m)+\pi_1(m,r) &= \left\lfloor \frac{5m-5}{6} \right\rfloor -1 + \left\lfloor \frac{5r-m+1}{6} \right\rfloor\\
&\ge \left\lfloor \frac{5r +4m -9}{6} \right\rfloor -1\\
&\ge \left\lfloor \frac{5r-5}{6} \right\rfloor + \left\lfloor \frac{4m-4}{6} \right\rfloor - 1\\
&\ge \left\lfloor \frac{5r-5}{6} \right\rfloor + 1\\
&\ge \phi(r) + 1.
\end{align}

We therefore have a 5-adic increase in the valuation of each component of $U^{(1)}\left( f \right)$ except for the coefficient of $\displaystyle\frac{y}{(1+5y)^{5n-4}}$.  If we remove this component from $U^{(1)}\left( f \right)$ and then divide by 5, what remains is indeed a member of $\mathcal{V}_{5n-4}^{(0)}$.

\end{proof}

Our last two theorems are very nearly sufficient to give us the 5-adic increase we need, with the notable exception of the components which contribute to the coefficient of $\displaystyle\frac{y}{(1+5y)^{5n-4}}$.  Indeed, the individual components need not be divisible by 5 at all.  We therefore need one additional condition for our purposes.

\begin{definition}\label{3termad}
\begin{align}
\mathcal{W}_n^{(1)}:=& \left\{ \frac{1}{(1+5y)^n}\sum_{m\ge 1} s(m)\cdot 5^{\theta(m)}\cdot y^m\in\mathcal{V}_{n}^{(1)} : \sum_{m=1}^3s(m)\equiv 0\bmod 5 \right\}.
\end{align}
\end{definition}

This small additional condition is at last sufficient for our purposes.  Not only do we need to prove that this allows 5-adic growth upon application of $U^{(1)}$, but that the condition is stable, i.e., that 

\begin{align*}
U^{(0)}\circ U^{(1)}\left(\mathcal{W}_n^{(1)}\right)&\subseteq \mathcal{W}_{25n-22}^{(1)}.
\end{align*}

\begin{theorem}\label{w12w1}
Suppose $f\in\mathcal{W}_n^{(1)}$.  Then

\begin{align}
\frac{1}{5}\left(U^{(1)}\left( f \right)\right) &\in\mathcal{V}_{5n-4}^{(0)},\label{wtov}\\
\frac{1}{5^2}\left(U^{(0)}\circ U^{(1)}\left( f \right)\right) &\in\mathcal{W}_{25n-22}^{(1)}\label{vtow}.
\end{align}

\end{theorem}

\begin{proof}
Let $f\in\mathcal{W}_n^{(1)}$ such that
\begin{align*}
f &= \frac{1}{(1+5y)^n}\sum_{m\ge 1} s(m)\cdot 5^{\theta(m)}\cdot y^m.
\end{align*}  We then have

\begin{align*}
U^{(1)}\left( f \right) &= \frac{1}{(1+5y)^{5n-4}}\sum_{m\ge 1}\sum_{r\ge \left\lceil m/5 \right\rceil} s(m)\cdot h_1(m,n,r)\cdot 5^{\theta(m) + \pi_1(m,r)}\cdot y^r\\
&=\frac{1}{(1+5y)^{5n-4}}t(1)\cdot 5^{\phi(1)}\cdot y + \frac{1}{(1+5y)^{5n-4}}\sum_{r\ge 2}t(r)\cdot 5^{\phi(r)+1}\cdot y^r,
\end{align*} with

\begin{align*}
t(r) &=\begin{cases}
\displaystyle{\sum_{1\le m\le 5}s(m)\cdot h_1(m,n,1)\cdot 5^{\theta(m) + \pi_1(m,1) - \phi(r)}},&\ r=1\\
\displaystyle{\sum_{1\le m\le 5r}s(m)\cdot h_1(m,n,r)\cdot 5^{\theta(m) + \pi_1(m,r) - \phi(r) - 1}},&\ r\ge 2.
\end{cases}
\end{align*}  We first prove (\ref{wtov}).  Notice that 

\begin{align*}
t(1) &= \sum_{m=1}^{5}s(m)\cdot h_1(m,n,1)\cdot 5^{\theta(m) + \pi_1(m,1)},
\end{align*} since $\phi(1) = 0$.  Moreover, $\theta(4),\theta(5)\ge 1$, and $\theta(m) + \pi_1(m,1)=0$ for $1\le m\le 3$, so that

\begin{align*}
t(1) &\equiv \sum_{m=1}^{3}s(m)\cdot h_1(m,n,1)\pmod{5}.
\end{align*}  Taking advantage of (\ref{hmod1}), we have

\begin{align*}
t(1)\equiv \sum_{m=1}^{3}s(m)\equiv 0\pmod 5.
\end{align*}  If we were now to write

\begin{align*}
\tilde{t}(r) := \begin{cases}
&\frac{1}{5}\cdot t(1)\in\mathbb{Z},\ r=1\\
&t(r),\ r\neq 1,
\end{cases}
\end{align*} then we have

\begin{align*}
U^{(1)}\left( f \right) &=\frac{1}{(1+5y)^{5n-4}}\sum_{r\ge 1}\tilde{t}(r)\cdot 5^{\phi(r)+1}\cdot y^r,
\end{align*} so that

\begin{align*}
\frac{1}{5}\left(U^{(1)}\left( f \right)\right) \in\mathcal{V}_{5n-4}^{(0)}.
\end{align*}

We now prove (\ref{vtow}).  Taking $U^{(0)}$ and dividing by $5^2$, we find

\begin{align*}
\frac{1}{5^2}\cdot\left(U^{(0)}\circ U^{(1)}\left( f \right)\right) =& \frac{5^{-2}}{(1+5y)^{25n-22}}\\ &\times\sum_{r\ge 1}\sum_{w\ge\left\lceil (r+2)/5 \right\rceil}\tilde{t}(r)\cdot h_0(r,5n-4,w)\cdot 5^{\pi_0(r,w)+\phi(r)}\cdot y^w\\
=&\frac{1}{(1+5y)^{25n-22}}\sum_{w\ge 1}q(r)\cdot 5^{\theta(w)} y^w,
\end{align*} with

\begin{align*}
q(w) =& \sum_{r=1}^{5w-2}\tilde{t}(r)\cdot h_0(r,5n-4,w)\cdot 5^{\pi_0(r,w)-\theta(w)-2}\\
=& \sum_{r=1}^{5w-2}\sum_{m=1}^{5r}s(m)\cdot h_1(m,n,r)\cdot h_0(r,5n-4,w)\cdot 5^{\theta(m) + \pi_1(m,r) + \pi_0(r,w) -\theta(w) -2}.
\end{align*}  In particular, since $\theta(w)=0$ for $1\le w\le 3$,

\begin{align*}
q(1) =& \sum_{r=1}^{3}\sum_{m=1}^{5r}s(m)\cdot h_1(m,n,r)\cdot h_0(r,5n-4,1)\cdot 5^{\theta(m) + \pi_1(m,r) + \pi_0(r,1) -2},\\
q(2) =& \sum_{r=1}^{8}\sum_{m=1}^{5r}s(m)\cdot h_1(m,n,r)\cdot h_0(r,5n-4,2)\cdot 5^{\theta(m) + \pi_1(m,r) + \pi_0(r,2) -2},\\
q(3) =& \sum_{r=1}^{13}\sum_{m=1}^{5r}s(m)\cdot h_1(m,n,r)\cdot h_0(r,5n-4,3)\cdot 5^{\theta(m) + \pi_1(m,r) + \pi_0(r,3) -2}.
\end{align*}  We want to show that $q(1)+q(2)+q(3)\equiv 0\pmod{5}.$  In the first place, we may remove all cases in which $\theta(m) + \pi_1(m,r) + \pi_0(r,w) -2 \ge 1$.  A quick estimation shows that

\begin{align*}
&\theta(m) + \pi_1(m,r) + \pi_0(r,w) -2\\ &\ge \left\lfloor \frac{5m-5}{6} \right\rfloor - 1 + \left\lfloor \frac{5r-m+1}{6} \right\rfloor + \left\lfloor \frac{5r-w-2}{6} \right\rfloor - 2\\
&\ge \left\lfloor \frac{5r+4m-9}{6} \right\rfloor + \left\lfloor \frac{5r-w-2}{6} \right\rfloor - 3\\
&\ge 2\left\lfloor \frac{5r-5}{6} \right\rfloor + \left\lfloor \frac{4m-4}{6} \right\rfloor - 3\\
&\ge 1,
\end{align*} for $m\ge 7$ or $r\ge 4$.  We therefore need to examine the cases for $1\le m\le 6$ and $1\le r\le 3$.  We provide three tables in Appendix I which compute $\theta(m) + \pi_1(m,r) + \pi_0(r,w) -2 \ge 1$ over this range.

Examining Table \ref{tablew1}, we see that we get a value of 0 for

\begin{align*}
(r,m)=(1,4),(2,1),(2,2),(3,3).
\end{align*}  Moreover, we get a value of $-1$ for 
\begin{align*}
(r,m)=(1,1),(1,2),(1,3).
\end{align*}

Examining Table \ref{tablew2}, we see that we get a value of 0 for
\begin{align*}
(r,m)=(1,4),(2,1),(2,2),(3,3),
\end{align*} and a value of $-1$ for 
\begin{align*}
(r,m)=(1,1),(1,2),(1,3).
\end{align*}

Finally, examining Table \ref{tablew3}, we see that we get a value of 0 for
\begin{align*}
(r,m)=(1,1),(1,2),(1,3),(2,1),(2,2),
\end{align*} and no negative values.

Taking $q(1)+q(2)+q(3)\pmod{5}$, we therefore have

\begin{align*}
q(1)&+q(2)+q(3)\\
\equiv &\frac{1}{5}\cdot\sum_{m=1}^{3}s(m)\cdot h_1(m,n,1)\cdot h_0(1,5n-4,1) + s(4)\cdot h_1(4,n,1)\cdot h_0(1,5n+4,1)\\ +& \sum_{m=1}^{2}s(m)\cdot h_1(m,n,2)\cdot h_0(2,5n-4,1) + s(3)\cdot h_1(3,n,3)\cdot h_0(3,5n-4,1)\\
+&\frac{1}{5}\cdot\sum_{m=1}^{3}s(m)\cdot h_1(m,n,1)\cdot h_0(1,5n-4,2)
+ s(4)\cdot h_1(4,n,1)\cdot h_0(1,5n+4,2)\\
+& \sum_{m=1}^{2}s(m)\cdot h_1(m,n,2)\cdot h_0(2,5n-4,2)
+ s(3)\cdot h_1(3,n,3)\cdot h_0(3,5n-4,2)\\
+& \sum_{m=1}^{3}s(m)\cdot h_1(m,n,1)\cdot h_0(1,5n-4,3) + \sum_{m=1}^{2}s(m)\cdot h_1(m,n,2)\cdot h_0(2,5n-4,3)\pmod{5}.
\end{align*}  Rearranging, we have

\begin{align}
q(1)&+q(2)+q(3)\nonumber\\
\equiv &\frac{1}{5}\cdot \left(\sum_{j=1}^2 h_0(1,5n-4,j) \right)\cdot\left( \sum_{m=1}^{3}s(m)\cdot h_1(m,n,1)\right)\label{q1}\\ +& h_0(1,5n-4,3)\cdot\left(\sum_{m=1}^{3}s(m)\cdot h_1(m,n,1)\right)\label{q2} \\ +& \left( \sum_{j=1}^2 h_0(1,5n-4,j)\right)\cdot s(4)\cdot h_1(4,n,1) \label{q3}\\ +& \left( \sum_{j=1}^3 h_0(2,5n-4,j) \right)\cdot \sum_{m=1}^{2}s(m)\cdot h_1(m,n,2)\label{q4}\\ +& \left( \sum_{j=1}^2 h_0(3,5n-4,j)\right)\cdot s(3)\cdot h_1(3,n,3) \pmod{5}.\label{q5}
\end{align}  It now remains to demonstrate that this expression is $0\pmod 5$.

We have placed parentheses around each sum which is divisible by 5.  In the first place, $h_1(m,n,1)\equiv 1\pmod{5}$ by (\ref{hmod1}).  Therefore,

\begin{align*}
\sum_{m=1}^{3}s(m)\cdot h_1(m,n,1)\equiv \sum_{m=1}^{3}s(m)\equiv 0\pmod 5,
\end{align*} since $f\in\mathcal{W}_n^{(1)}$.  Moreover,

\begin{align*}
\sum_{j=1}^2 h_0(1,5n-4,j)\equiv 0\pmod{5}
\end{align*} by (\ref{hmod2}), (\ref{hmod2a}).  Therefore, (\ref{q1}) is $0\pmod 5$.  

In like manner, we have the parenthesized sums in 
(\ref{q2}) equivalent to $0\pmod 5$ by (\ref{hmod1});
(\ref{q3}) equivalent to $0\pmod 5$ by (\ref{hmod2}) and (\ref{hmod2a});
(\ref{q4}) equivalent to $0\pmod 5$ by (\ref{hmod3}), (\ref{hmod3b}), and (\ref{hmod3c});
(\ref{q5}) equivalent to $0\pmod 5$ by (\ref{hmod4}) and (\ref{hmod4a}).

We then have 

\begin{align*}
\frac{1}{5^2}\cdot\left(U^{(0)}\circ U^{(1)}\left( f \right)\right) =&\frac{1}{(1+5y)^{25n-22}}\sum_{w\ge 1}q(r)\cdot 5^{\theta(w)} y^w\in\mathcal{V}_{25n-22}^{(1)},
\end{align*} with $q(1)+q(2)+q(3)\equiv 0\pmod{5}$, i.e.,

\begin{align}
\frac{1}{5^2}\cdot\left(U^{(0)}\circ U^{(1)}\left( f \right)\right)\in\mathcal{W}_{25n-22}^{(1)}.
\end{align}

\end{proof}

At last, we have enough to prove Theorem \ref{Mythm}:

\begin{proof}[Proof of Theorem \ref{Mythm}]

In Section 5 we will demonstrate that

\begin{align*}
L_1 =& \frac{F}{(1 + 5 y)^3}\cdot\left( 120 y + 1805 y^2 + 12050 y^3 + 39500 y^4 + 50000 y^5\right)\\
=&\frac{5\cdot F}{(1 + 5 y)^3}\cdot\left( 24 y + 361 y^2 + 2410 y^3 + 7900 y^4 + 10000 y^5\right).
\end{align*}  Notice that

\begin{align*}
\frac{1}{5\cdot F}\cdot L_1 = f_1\in\mathcal{W}_{3}^{(1)}.
\end{align*}  Suppose that for some $\alpha\in\mathbb{Z}_{\ge 1}$, we have
\begin{align*}
\frac{1}{5^{2\alpha-1}\cdot F}\cdot L_{2\alpha-1}\in\mathcal{W}_n^{(1)},
\end{align*} for some $n\in\mathbb{Z}_{\ge 1}$.  Then

\begin{align}
L_{2\alpha-1} = F\cdot 5^{2\alpha-1}\cdot f_{2\alpha-1},
\end{align} with $f_{2\alpha-1}\in\mathcal{W}_n^{(1)}$.  Now,

\begin{align}
L_{2\alpha} = U_5\left(L_{2\alpha-1} \right) = U_5\left( F\cdot 5^{2\alpha-1}\cdot f_{2\alpha-1} \right) = F\cdot 5^{2\alpha-1}\cdot U^{(1)}\left( f_{2\alpha-1} \right).
\end{align}  By (\ref{wtov}) of Theorem \ref{w12w1}, we know that there exists some $f_{2\alpha}\in\mathcal{V}_{5n-4}^{(0)}$ such that

\begin{align}
U^{(1)}\left( f_{2\alpha-1} \right) = 5\cdot f_{2\alpha}.
\end{align} Therefore

\begin{align}
L_{2\alpha} = F\cdot 5^{2\alpha}\cdot f_{2\alpha}.
\end{align}  Moreover,

\begin{align}
L_{2\alpha+1} = U_5\left( Z\cdot L_{2\alpha} \right) = U_5\left( F\cdot 5^{2\alpha}\cdot Z\cdot f_{2\alpha}\right) = F\cdot 5^{2\alpha}\cdot U^{(0)}\left( f_{2\alpha} \right).
\end{align}  By (\ref{vtow}) of Theorem \ref{w12w1}, we know that there exists some $f_{2\alpha+1}\in\mathcal{W}_{5n-2}^{(1)}$ such that

\begin{align}
U^{(0)}\left( f_{2\alpha} \right) = 5\cdot f_{2\alpha+1}.
\end{align}  Therefore,

\begin{align}
L_{2\alpha+1} = F\cdot 5^{2\alpha+1}\cdot f_{2\alpha+1}.
\end{align}

We briefly show that the power of our localizing factor for $L_{\alpha}$ matches with $\psi(\alpha)$, i.e., that

\begin{align*}
\frac{L_{2\alpha-1}}{5^{2\alpha-1}\cdot F}&\in\mathcal{W}_{\psi(2\alpha-1)}^{(1)},\\
\frac{L_{2\alpha}}{5^{2\alpha}\cdot F}&\in\mathcal{V}_{\psi(2\alpha)}^{(0)}.
\end{align*}  It is a fact of elementary number theory that for all $\alpha\ge 1$,

\begin{align*}
5^{2\alpha-1}&\equiv 5\pmod{12},\\
5^{2\alpha}&\equiv 1\pmod{12},
\end{align*} and therefore that

\begin{align*}
\left\lfloor \frac{5^{2\alpha-1}}{12} \right\rfloor &= \frac{5^{2\alpha-1}}{12} - \frac{5}{12},\\
\left\lfloor \frac{5^{2\alpha}}{12} \right\rfloor &= \frac{5^{2\alpha}}{12} - \frac{1}{12}.
\end{align*}  With this, we have

\begin{align*}
5\cdot\psi(2\alpha-1)-4 &= 5\cdot \left(\left\lfloor \frac{5^{2\alpha}}{12} \right\rfloor + 1\right)-4\\
&= 5\cdot \left(\frac{5^{2\alpha}}{12} - \frac{1}{12} + 1\right)-4\\
&= \frac{5^{2\alpha+1}}{12} - \frac{5}{12} + 1\\
&= \left\lfloor \frac{5^{2\alpha+1}}{12} \right\rfloor + 1\\
&= \psi(2\alpha).
\end{align*}  In similar fashion, it can be proved that

\begin{align*}
5\cdot\psi(2\alpha)-2 = \psi(2\alpha+1).
\end{align*}  This is compatible with the increase in the localizing powers in Theorem \ref{thmuyox}.  Finally, $\psi(1) = 3$ is the localizing power for $L_1$.

\end{proof}

\section{Initial Relations}

For $i$ fixed, our theorem for expanding $U^{(i)}\left( \frac{y^{m}}{(1+5y)^{n}} \right)$ requires 25 initial relations to be justified.  However, these relations are ultimately dependent on far fewer relations, since one can very quickly verify that

\begin{align}
U^{(i)}\left( \frac{y^{m}}{(1+5y)^{n}} \right) =& \frac{1}{5^m}\cdot U^{(i)}\left( \frac{(x-1)^{m}}{x^{n}} \right)\\
=& \frac{1}{5^m}\sum_{r=0}^{m}(-1)^{m-r}{{m}\choose{r}}\cdot U^{(i)}\left( x^{r-n} \right)\\
=& \frac{1}{5^m}\sum_{r=0}^{m}(-1)^{m-r}{{m}\choose{r}}\cdot U^{(i)}\left( (1+5y)^{r-n} \right).
\end{align}  We can very quickly compute any value of $U^{(i)}\left( \frac{y^{m}}{(1+5y)^{n}} \right)$, provided we have exact expressions for $U^{(i)}\left( (1+5y)^{r} \right)$ for $-n\le r\le m-n$.

To compute $U^{(i)}\left( \frac{y^{m}}{(1+5y)^{n}} \right)$ for $1\le m,n\le 5$, we need to have expressions for $U^{(i)}\left( (1+5y)^{r} \right)$ for $-5\le r\le 4$.  However, we have the degree 5 modular equation for $x=1+5y$, which yields

\begin{align}
U^{(i)}\left( (1+5y)^n \right) &= -\sum_{k=0}^{4} b_k(\tau)\cdot U^{(i)}\left( (1+5y)^{k+n-5} \right).
\end{align}  Moreover, for $n\ge 0$ we obviously have

\begin{align}
U^{(i)}\left( (1+5y)^n \right) &= \sum_{k=0}^n {{n}\choose{k}}\cdot 5^k\cdot U^{(i)}\left( y^k \right),
\end{align} and $y$ follows a degree 5 modular equation.

Therefore, in order to determine $U^{(i)}\left( \frac{y^{m}}{(1+5y)^{n}} \right)$ for any $m,n\in\mathbb{Z}$ and $i$ fixed at 0 or 1, we only need to determine relations for $U^{(i)}\left( y^k \right)$ for five consecutive values of $k$.  For both values of $i$, that makes 10 relations.  Once these relations are established, verification of the 50 initial relations follows as a relatively simple, if somewhat tedious, computational exercise.

\begin{theorem}
The relations from Theorem \ref{thmuyox}, together with the congruence conditions of Corollary \ref{congcor1}, hold for $1\le m\le 5,$ and $1\le n\le 5$.
\end{theorem}  The calculation is straightforward, but we detail it in the Mathematica supplement to this paper, which can be found online at \url{https://www3.risc.jku.at/people/nsmoot/online3.nb}.

We choose to justify these ten relations using finiteness conditions from the theory of modular functions.

This approach is useful, given that the right-hand sides of eight of the ten relations below are already modular functions with a pole only at a single cusp of $\mathrm{X}_0(10)$ (the remaining two can become such functions with a slight adjustment).  The remaining task is merely to verify that the left-hand side of each relation is also a modular function with a pole at the same cusp, and then to compare the principal parts from either side.  Because it is slightly easier to expand both sides of each relation with respect to the pole at $\infty$, we will divide by a sufficient power of $y$ to induce a pole at $\infty$ rather than at 0.

We will provide some, though not all, details to our cusp analysis computations here.  The full computations are detailed in our online Mathematica supplement.

We preface our results with a brief overview of the theory of modular forms.  For a classic review of the theory, see \cite{Knopp}.  For a more modern treatment, see \cite{Diamond}.

\subsection{Preliminaries}
We denote $\mathbb{H}$ as the upper half complex plane, and\\ \noindent$\hat{\mathbb{H}}:=\mathbb{H}\cup\{\infty\}\cup\mathbb{Q}$.  We also define $\hat{\mathbb{Q}}:=\mathbb{Q}\cup\{\infty\}$, with $a/0=\infty$ for any $a\neq 0$.

We denote $\mathrm{SL}(2,\mathbb{Z})$ as the set of all $2\times 2$ integer matrices with determinant~1.

For any given $N\in\mathbb{Z}_{\ge 1}$, let

\begin{align*}
\Gamma_0(N) = \Bigg\{ \begin{pmatrix}
  a & b \\
  c & d 
 \end{pmatrix}\in \mathrm{SL}(2,\mathbb{Z}) : N|c \Bigg\}.
\end{align*}

We define a group action

\begin{align*}
\Gamma_0(N)\times\hat{\mathbb{H}}&\longrightarrow\hat{\mathbb{H}},\\
\left(\begin{pmatrix}
  a & b \\
  c & d 
 \end{pmatrix},\tau\right)&\longrightarrow \frac{a\tau+b}{c\tau+d}.
\end{align*} If $\gamma=\begin{pmatrix}
  a & b \\
  c & d 
 \end{pmatrix}$ and $\tau\in\hat{\mathbb{H}}$, then we write
 
 \begin{align*}
 \gamma\tau := \frac{a\tau+b}{c\tau+d}.
 \end{align*}  The orbits of this action are defined as

\begin{align*}
[\tau]_N := \left\{ \gamma\tau: \gamma\in\Gamma_0(N) \right\}.
\end{align*}

\begin{definition}
For any $N\in\mathbb{Z}_{\ge 1}$, we define the classical modular curve of level $N$ as the set of all orbits of $\Gamma_0(N)$ applied to $\hat{\mathbb{H}}$:

\begin{align*}
\mathrm{X}_0(N):=\left\{ [\tau]_N : \tau\in\hat{\mathbb{H}} \right\}
\end{align*}
\end{definition}

The group action applied to $\hat{\mathbb{H}}$ can be restricted to $\hat{\mathbb{Q}}$: that is, for every $\tau\in\hat{\mathbb{Q}}$, $[\tau]_N\subseteq\hat{\mathbb{Q}}$.  There are only a finite number of such orbits \cite[Section 3.8]{Diamond}.

\begin{definition}
For any $N\in\mathbb{Z}_{\ge 1}$, the cusps of $\mathrm{X}_0(N)$ are the orbits of $\Gamma_0(N)$ applied to $\hat{\mathbb{Q}}$.
\end{definition}

The detailed properties of $\mathrm{X}_0(N)$, including its Riemann surface structure, are given in \cite[Chapters 2,3]{Diamond}.  For want of space, we will only add that $\mathrm{X}_0(N)$ possesses a unique nonnegative integer $g\left(  \mathrm{X}_0(N) \right)$ called its genus.  This number is referenced in the Introduction, and may be computed using Theorem 3.1.1 of \cite[Chapter 3]{Diamond}.  For an understanding of the connection of the genus to module rank, e.g., via the Paule--Radu method, see Section 6 below.

\begin{definition}\label{DefnModular}
Let $f:\mathbb{H}\longrightarrow\mathbb{C}$ be holomorphic on $\mathbb{H}$.  Then $f$ is a weakly holomorphic modular form of weight $k\in\mathbb{Z}$ over $\Gamma_0(N)$ if the following properties are satisfied for every $\gamma=\left(\begin{smallmatrix}
  a & b \\
  c & d 
 \end{smallmatrix}\right)\in\mathrm{SL}(2,\mathbb{Z})$:

\begin{enumerate}
\item We have $$\displaystyle{f\left( \gamma\tau \right) = (c\tau+d)^{k} \sum_{n=n_{\gamma}}^{\infty}\alpha_{\gamma}(n)q^{n\gcd(c^2,N)/ N}},$$  with $n_{\gamma}\in\mathbb{Z}$, and $\alpha_{\gamma}(n_{\gamma})\neq 0$.  If $n_{\gamma}\ge 0$, then $f$ is holomorphic at the cusp $[a/c]_N$.  Otherwise, $f$ has a pole of order $n_{\gamma}$, and principal part
 \begin{align}
 \sum_{n=n_{\gamma}}^{-1}\alpha_{\gamma}(n)q^{n\gcd(c^2,N)/ N},\label{princpartmod}
 \end{align} at the cusp $[a/c]_N$.
 \item If $\gamma\in\Gamma_0(N)$, we have $\displaystyle{f\left( \gamma\tau \right) = (c\tau+d)^{k} f(\tau)}.$
\end{enumerate}  We refer to $\mathrm{ord}_{a/c}^{(N)}(f) := n_{\gamma}(f)$ as the order of $f$ at the cusp $[a/c]_N$.  If $f$ is holomorphic at every cusp, then $f$ is a holomorphic modular form.  If $k=0$, then $f$ is a modular function.
\end{definition}

We now define the relevant sets of all modular functions:

\begin{definition}
Let $\mathcal{M}_k\left(\Gamma_0(N)\right)$ be the set of all weight $k$ holomorphic modular forms over $\Gamma_0(N)$, and let $\mathcal{M}\left(\Gamma_0(N)\right)$ be the set of all modular functions over $\Gamma_0(N)$, and $\mathcal{M}^{a/c}\left(\Gamma_0(N)\right)\subset \mathcal{M}\left(\Gamma_0(N)\right)$ to be those modular functions over $\Gamma_0(N)$ with a pole only at the cusp $[a/c]_N$.  These are all commutative algebras with 1, and standard addition and multiplication \cite[Section 2.1]{Radu}.
\end{definition}

Due to its precise symmetry over $\Gamma_0(N)$, any modular function $f\in\mathcal{M}\left( \Gamma_0(N) \right)$ induces a well-defined function

\begin{align*}
\hat{f}&:\mathrm{X}_0(N)\longrightarrow\mathbb{C}\cup\{\infty\}\\
&:[\tau]_N\longrightarrow f(\tau).
\end{align*}  The notions of pole order and cusps of $f$ used in Definition \ref{DefnModular} have been constructed so as to coincide with these notions applied to $\hat{f}$ on $\mathrm{X}_0(N)$.  In particular, (\ref{princpartmod}) represents the principal part of $\hat{f}$ in local coordinates near the cusp $[a/c]_N$.  Notice that as $\tau\rightarrow i\infty$, we must have $\gamma\tau\rightarrow a/c$, and $q\rightarrow 0$.

Because $f$ is holomorphic on $\mathbb{H}$ by definition, any possible poles for $\hat{f}$ must be found for $[\tau]_N\subseteq\hat{\mathbb{Q}}$.  The number and order of these poles is of paramount importance to us.

We now give an extremely important result in the general theory of Riemann surfaces:

\begin{theorem}
Let $\mathrm{X}$ be a compact Riemann surface, and let $\hat{f}:\mathrm{X}\longrightarrow\mathbb{C}$ be analytic on all of $\mathrm{X}$.  Then $\hat{f}$ must be a constant function.
\end{theorem}

The importance of this theorem cannot be overstated.  As an immediate consequence we have

\begin{corollary}
For a given $N\in\mathbb{Z}_{\ge 1}$, if $f\in\mathcal{M}\left(\Gamma_0(N)\right)$ has no poles at any cusp of $\Gamma_0(N)$, then $f$ must be a constant.
\end{corollary}

This is immensely useful for verifying that two modular functions over the same space are equivalent.  For example, let us take $f,g\in\mathcal{M}^{\infty}\left(\Gamma_0(N)\right)$.  Then this means that $f,g$ both have principal parts only at a single cusp.

If these principal parts of each function match, then $f-g\in\mathcal{M}\left(\Gamma_0(N)\right)$ can have no poles at any cusp.  This implies that $\hat{f}-\hat{g}$ is analytic on the whole of $\mathrm{X}_0(N)$, which forces $\hat{f}-\hat{g}$, and therefore $f-g$, to be a constant.  If the constants of $f$ and $g$ also match, then $f-g=0$, i.e., $f=g$.

We now give two key theorems that will prove useful in checking the modularity of certain functions.  We will use Dedekind's eta function \cite[Chapter 3]{Knopp}:

\begin{align*}
\eta(\tau):= e^{\pi i\tau/12}\prod_{n=1}^{\infty}\left( 1-e^{2\pi i n\tau} \right).
\end{align*}

The first is a theorem by Newman \cite[Theorem 1]{Newman}:

\begin{theorem}
Let $f = \prod_{\delta | N} \eta(\delta\tau)^{r_{\delta}}$, with $\hat{r} = (r_{\delta})_{\delta | N}$ an integer-valued vector, for some $N\in\mathbb{Z}_{\ge 1}$.  Then $f\in\mathcal{M}\left(\Gamma_0(N)\right)$ if and only if the following apply:

\begin{enumerate}
\item $\sum_{\delta | N} r_{\delta} = 0;$
\item $\sum_{\delta | N} \delta r_{\delta} \equiv 0\pmod{24};$
\item $\sum_{\delta | N} \frac{N}{\delta}r_{\delta} \equiv 0\pmod{24};$
\item $\prod_{\delta | N} \delta^{|r_{\delta}|}$ is a perfect square.
\end{enumerate}

\end{theorem}

To study the order of an eta quotient at a given cusp, we make use of a theorem that can be found in \cite[Theorem 23]{Radu}, generally attributed to Ligozat:

\begin{theorem}
If $f = \prod_{\delta | N} \eta(\delta\tau)^{r_{\delta}}\in\mathcal{M}\left(\Gamma_0(N)\right)$, then the order of $f$ at the cusp $[a/c]_N$ is given by the following:

\begin{align*}
\mathrm{ord}_{a/c}^{(N)}(f) = \frac{N}{24\gcd{(c^2,N)}}\sum_{\delta | N} r_{\delta}\frac{\gcd{(c,\delta)}^2}{\delta}.
\end{align*}

\end{theorem}

\subsection{Computing the Initial Cases}

Despite its apparent recondite nature, the theory above imposes various finiteness conditions which can be used to a computational advantage.  Our initial relations consist of

\begin{align}
U^{(i)}\left( y^l \right) &=p_{i,l}(y)\in\mathbb{Z}[y],
\end{align} for $1\le l\le 4$, and
\begin{align}
(1+5y)\cdot U^{(i)}\left( 1 \right) &=p_{i,0}(y)\in\mathbb{Z}[y]
\end{align} (in both cases, $0\le i\le 1$).

We can use Newman's theorem to verify that $y\in\mathcal{M}(\Gamma_0(10))$, and Ligozat's theorem to show that 

\begin{align*}
\mathrm{ord}_{\infty}^{(10)}(y^{-1}) &= -1,\\
\mathrm{ord}_{1/5}^{(10)}(y^{-1}) &= 0,\\
\mathrm{ord}_{1/2}^{(10)}(y^{-1}) &= 0,\\
\mathrm{ord}_{0}^{(10)}(y^{-1}) &= 1.
\end{align*}  This proves that $1/y\in\mathcal{M}^{\infty}(\Gamma_0(10))$.  Therefore, if we denote $m$ as the degree of $p_{i,l}$, then multiplying both sides of the proposed relations above by $1/y^{m}$, our relations take on the form

\begin{align}
\frac{1}{y^{m}}\cdot U^{(i)}\left( y^l \right) &\in\mathbb{Z}[y^{-1}]\subseteq\mathcal{M}^{\infty}(\Gamma_0(10)),\label{initialTypeA}
\end{align}
\begin{align}
\frac{1}{y^{m}}\cdot (1+5y)\cdot U^{(i)}\left( 1 \right) &\in\mathbb{Z}[y^{-1}]\subseteq\mathcal{M}^{\infty}(\Gamma_0(10)).\label{initialTypeB}
\end{align}  Here all that remains is to verify that the left-hand sides of each prospective relation are elements of $\mathcal{M}^{\infty}(\Gamma_0(10))$.  Then we may compare the principal parts and constants of both sides: equality of these parts implies equality overall.

We will begin with the relations of the form (\ref{initialTypeA}).  If we recall the definition of $U^{(i)}$, then our left-hand side takes the form

\begin{align*}
&\frac{1}{y(\tau)^m}\cdot\frac{1}{F(\tau)}\cdot U_5\left( F(\tau)\cdot Z(\tau)^{1-i}\cdot y(\tau)^l \right)\\
&=U_5\left( \frac{F(\tau)}{F(5\tau)}\cdot \frac{Z(\tau)^{1-i}\cdot y(\tau)^l}{y(5\tau)^m} \right).
\end{align*}  Now, it is well-known (e.g., \cite[Corollary 2.3]{Garvan}) that

\begin{align*}
F(\tau)\in\mathcal{M}_2(\Gamma_0(2))\subseteq\mathcal{M}_2(\Gamma_0(10))\subseteq\mathcal{M}_2(\Gamma_0(50)).
\end{align*}  Moreover,

\begin{align*}
F(5\tau)&\in\mathcal{M}_2\left( \Gamma_0(50) \right).
\end{align*}  This implies that

\begin{align*}
\frac{F(\tau)}{F(5\tau)}&\in\mathcal{M}\left( \Gamma_0(50) \right).
\end{align*}  One can directly compute that

\begin{align*}
\mathrm{ord}_{\infty}^{(50)}\left( F\left( \tau \right) \right) &= 1,& \mathrm{ord}_{\infty}^{(50)}\left( F\left( 5\tau \right) \right) = 5,\\
\mathrm{ord}_{1/25}^{(50)}\left( F\left( \tau \right) \right) &= 0,& \mathrm{ord}_{1/25}^{(50)}\left( F\left( 5\tau \right) \right) = 0,\\
\mathrm{ord}_{1/2}^{(50)}\left( F\left( \tau \right) \right) &= 5,& \mathrm{ord}_{1/2}^{(50)}\left( F\left( 5\tau \right) \right) = 1,\\
\mathrm{ord}_{0}^{(50)}\left( F\left( \tau \right) \right) &= 5,& \mathrm{ord}_{0}^{(50)}\left( F\left( 5\tau \right) \right) = 1,\\
\mathrm{ord}_{k/5}^{(50)}\left( F\left( \tau \right) \right) &= 1,&\mathrm{ord}_{k/5}^{(50)}\left( F\left( 5\tau \right) \right) = 1\ (k=1,2,3,4),\\
\mathrm{ord}_{k/10}^{(50)}\left( F\left( \tau \right) \right) &= 1,&\mathrm{ord}_{k/10}^{(50)}\left( F\left( 5\tau \right) \right) = 1\ (k=1,3,7,9).
\end{align*}  From this, we have 

\begin{align*}
\mathrm{ord}_{\infty}^{(50)}\left( \frac{F\left( \tau \right)}{F(5\tau)} \right) &= -4,\\
\mathrm{ord}_{1/25}^{(50)}\left( \frac{F\left( \tau \right)}{F(5\tau)} \right) &= 0,\\
\mathrm{ord}_{1/2}^{(50)}\left( \frac{F\left( \tau \right)}{F(5\tau)} \right) &= 4,\\
\mathrm{ord}_{0}^{(50)}\left( \frac{F\left( \tau \right)}{F(5\tau)} \right) &= 4,\\
\mathrm{ord}_{k/5}^{(50)}\left( \frac{F\left( \tau \right)}{F(5\tau)} \right) &= 0,\ (k=1,2,3,4),\\
\mathrm{ord}_{k/10}^{(50)}\left( \frac{F\left( \tau \right)}{F(5\tau)} \right) &= 0,\ (k=1,3,7,9).\\
\end{align*}  Therefore, we have

\begin{align*}
\frac{F(\tau)}{F(5\tau)}\in\mathcal{M}^{\infty}(\Gamma_0(50)),
\end{align*} with a zero of order 4 at the cusp $[1/2]_{50}$.  On the other hand, if we take

\begin{align*}
W_{i,l,m}(\tau):=\frac{Z(\tau)^{1-i}\cdot y(\tau)^l}{y(5\tau)^m},
\end{align*} we can use Ligozat's theorem to show that

\begin{align*}
\mathrm{ord}_{0}^{(50)}\left( W_{i,l,m}(\tau) \right) &= -(1-i) - 5l + m,\\
\mathrm{ord}_{1/2}^{(50)}\left( W_{i,l,m}(\tau) \right) &= -2(1-i),\\
\mathrm{ord}_{k/5}^{(50)}\left( W_{i,l,m}(\tau) \right) &= m\ (k=1,2,3,4),\\
\mathrm{ord}_{k/10}^{(50)}\left( W_{i,l,m}(\tau) \right) &= l\ (k=1,3,7,9).
\end{align*}  Therefore, if we let

\begin{align*}
G_{i,l,m}(\tau):=\frac{F(\tau)W_{i,l,m}(\tau)}{F(5\tau)},
\end{align*} then we have

\begin{align*}
\mathrm{ord}_{0}^{(50)}\left( G_{i,l,m}(\tau) \right) &= 4 -(1-i) - 5l + m,\\
\mathrm{ord}_{1/2}^{(50)}\left( G_{i,l,m}(\tau) \right) &= 4 -2(1-i),\\
\mathrm{ord}_{k/5}^{(50)}\left( G_{i,l,m}(\tau) \right) &= m\ (k=1,2,3,4),\\
\mathrm{ord}_{k/10}^{(50)}\left( G_{i,l,m}(\tau) \right) &= l\ (k=1,3,7,9).
\end{align*}  Inspection of the relations in Appendix II quickly reveals that $4 -(1-i) - 5l + m = 0$ for each relation of the form (\ref{initialTypeA}).  Moreover, $4 -2(1-i)$ is clearly positive for $i=0,1$.  The final two orders are positive, as $m,l >0$.

For relations of the form (\ref{initialTypeB}), our right-hand side has the form

\begin{align*}
&\frac{1}{y(\tau)^2}\cdot x(\tau)\cdot\frac{1}{F(\tau)}\cdot U_5\left( F(\tau)\cdot Z(\tau)^{1-i}\right)\\
&=U_5\left( \frac{F(\tau)}{F(5\tau)}\cdot \frac{Z(\tau)^{1-i}\cdot x(5\tau)}{y(5\tau)^2} \right).
\end{align*}  If we take

\begin{align*}
W_{i}(\tau):=\frac{Z(\tau)^{1-i}\cdot x(5\tau)}{y(5\tau)^2},
\end{align*} we can use Ligozat's theorem to show that

\begin{align*}
\mathrm{ord}_{0}^{(50)}\left( W_{i}(\tau) \right) &= 1-i,\\
\mathrm{ord}_{1/2}^{(50)}\left( W_{i}(\tau) \right) &= 1-2i,\\
\mathrm{ord}_{k/5}^{(50)}\left( W_{i}(\tau) \right) &= 1\ (k=1,2,3,4),\\
\mathrm{ord}_{k/10}^{(50)}\left( W_{i}(\tau) \right) &= 1\ (k=1,3,7,9).
\end{align*}  Therefore, if we let 

\begin{align*}
G_{i}(\tau):=\frac{F(\tau)W_{i,l,m}(\tau)}{F(5\tau)},
\end{align*} then we have

\begin{align*}
\mathrm{ord}_{0}^{(50)}\left( G_{i}(\tau) \right) &= 5-i,\\
\mathrm{ord}_{1/2}^{(50)}\left( G_{i}(\tau) \right) &= 5-2i,\\
\mathrm{ord}_{k/5}^{(50)}\left( G_{i}(\tau) \right) &= 1\ (k=1,2,3,4),\\
\mathrm{ord}_{k/10}^{(50)}\left( G_{i}(\tau) \right) &= 1\ (k=1,3,7,9).
\end{align*}  These orders are all nonnegative.

The functions inside the $U_5$ operators on the right hand sides of each of our prospective relations are each elements of $\mathcal{M}^{\infty}(\Gamma_0(50))$.  Therefore, the $U_5$ operator pushes each to an element of $\mathcal{M}^{\infty}(\Gamma_0(50/5))=\mathcal{M}^{\infty}(\Gamma_0(10))$ \cite[Lemma 4.4]{Paule4}.

We have verified that the left hand side of each of our relation can be be put into a relation of the form (\ref{initialTypeA}) or (\ref{initialTypeB}), in which either side is an element of $\mathcal{M}^{\infty}(\Gamma_0(10))$.  All that remains is to examine the principal parts and constants of each of these relations.

This approach can also be used to prove (\ref{L1S}).  In this case, we want to prove that

\begin{align}
U_5&\left( \frac{L_0(\tau)}{F(5\tau)}\cdot \frac{Z(\tau)\cdot x(5\tau)^3}{y(5\tau)^5} \right)\nonumber\\ &= \left( 120 y^{-4} + 1805 y^{-3} + 12050 y^{-2} + 39500 y^{-1} + 50000 \right).\label{L1overy}
\end{align}  If we define

\begin{align*}
W_y := \frac{Z(\tau)\cdot x(5\tau)^3}{y(5\tau)^5},
\end{align*} then

\begin{align*}
\mathrm{ord}_{0}^{(50)}\left( W_{y}(\tau) \right) &= 1,\\
\mathrm{ord}_{1/2}^{(50)}\left( W_{y}(\tau) \right) &= 1,\\
\mathrm{ord}_{k/5}^{(50)}\left( W_{y}(\tau) \right) &= 2\ (k=1,2,3,4),\\
\mathrm{ord}_{k/10}^{(50)}\left( W_{y}(\tau) \right) &= 3\ (k=1,3,7,9).
\end{align*}  If we let 

\begin{align*}
G_{y}(\tau):=\frac{W_{y}(\tau)}{F(5\tau)},
\end{align*} then

\begin{align*}
\mathrm{ord}_{0}^{(50)}\left( G_{y}(\tau) \right) &\ge -1+1,\\
\mathrm{ord}_{1/2}^{(50)}\left( G_{y}(\tau) \right) &\ge -1+1,\\
\mathrm{ord}_{k/5}^{(50)}\left( G_{y}(\tau) \right) &\ge -1+2\ (k=1,2,3,4),\\
\mathrm{ord}_{k/10}^{(50)}\left( G_{y}(\tau) \right) &\ge -1+2\ (k=1,3,7,9).
\end{align*}  These orders are again all nonnegative.  Because $L_0\in\mathcal{M}_2(\Gamma_0(10))$, it will contribute no poles, and we need not examine it.  In this case, the principal part on either side of (\ref{L1overy}) takes the form

\begin{align*}
\frac{120}{q^4} + \frac{365}{q^3} + \frac{2765}{q^2} + \frac{5030}{q} + 9375.
\end{align*}

As a final application, we consider the proof of (\ref{modY}).  We can use Ligozat's theorem to determine that $y(5\tau)^{-1}\in\mathcal{M}^{\infty}\left(\Gamma_0(50)\right)$, and that $y(5\tau)^{-5}\cdot y(\tau)\in\mathcal{M}^{\infty}\left(\Gamma_0(50)\right)$.  As such, the principal part and constant of

\begin{align}
y(5\tau)^{-25}\cdot\left(y^5+\sum_{j=0}^4 a_j(5\tau) y^j\right)
\end{align} can quickly be verified to be 0, thus giving us (\ref{modY}).

\section{Necessity of the Paule--Radu Technique For Genus 1 Surfaces}

The proof strategy used above was able to give a single-variable proof of Theorem \ref{Thm12}.  We certainly believe that such a strategy can be extended to other families of congruences over any genus 0 curve.  It is tempting to try a similar proof for any family of congruences over a genus 1 curve in which the standing proof employs the Paule--Radu technique, e.g., the proof of the Andrews--Sellers theorem.  However, this is very likely a hopeless endeavor, and we will try to briefly explain why.  For a more in-depth discussion of the relationship between genus and the rank of the $\mathbb{Z}[X]$ modules of interest to us, see \cite{Paule3}.

Wang and Yang define \cite{Wang2} the functions

\begin{align*}
\rho &= \frac{\eta(2\tau)^2\eta(5\tau)^4}{\eta(\tau)^4\eta(10\tau)^2}\in\mathcal{M}^{0}\left( \Gamma_0(10) \right),& t = \frac{\eta(5\tau)^2\eta(10\tau)^2}{\eta(\tau)^2\eta(2\tau)^2}\in\mathcal{M}\left( \Gamma_0(10) \right).
\end{align*}  They go on to demonstrate that for all $\alpha \ge 1$,

\begin{align*}
\frac{L_{\alpha}}{F}\in \left< t,\rho t \right>_{\mathbb{Z}[t]}.
\end{align*}  This is indicative of the machinery of the Paule--Radu technique, which has been successfully applied in cases such as proving the Andrews--Sellers conjecture \cite{Paule} and the Choi--Kim--Lovejoy conjecture \cite{Smoot}.  These latter problems could not be solved by the classical techniques developed by Watson in 1938 (and indeed, known by Ramanujan himself some decades prior), in which the relevant functions $L_{\alpha}$ are understood as elements of $\left< 1 \right>_{\mathbb{Z}[t]} = \mathbb{Z}[t]$, for a given modular function $t$.

In the case of the Andrews--Sellers conjecture, the necessary spaces of functions are free rank 2 $\mathbb{Z}[X]$-modules.  Why does the Andrews--Sellers conjecture necessitate this more complex module structure while the functions associated with $\mathrm{spt}_{\omega}(n)$) can be described with the simpler localized ring $\mathbb{Z}[X]_{\mathcal{S}}$?

To understand this, consider the order of the functions $t, \rho$ over the cusps of $\mathrm{X}_0(10)$:

\begin{align*}
&\mathrm{ord}^{(10)}_{\infty}(t) = 1, \ \ \ \ \ \ \ \  \mathrm{ord}^{(10)}_{\infty}(\rho) = 0, \ \ \ \ \ \ \mathrm{ord}^{(10)}_{\infty}(x) = 0\\
&\mathrm{ord}^{(10)}_{0}(t) = -1, \ \ \ \ \  \mathrm{ord}_{0}^{(10)}(\rho) = -1, \ \ \ \ \  \mathrm{ord}^{(10)}_{0}(x) = -1\\
&\mathrm{ord}^{(10)}_{1/2}(t) = -1, \ \ \ \ \mathrm{ord}_{1/2}^{(10)}(\rho) = 0, \ \ \ \ \ \ \mathrm{ord}^{(10)}_{1/2}(x) = 1\\
&\mathrm{ord}^{(10)}_{1/5}(t) = 1, \ \ \ \ \ \ \  \mathrm{ord}^{(10)}_{1/5}(\rho) = 1, \ \ \ \ \ \  \mathrm{ord}^{(10)}_{1/5}(x) = 0\\
\end{align*}

Notice that $xt, \rho\in\mathcal{M}^{0}(10)$, with $\mathrm{ord}^{(10)}_0(xt) = -2$, $\mathrm{ord}^{(10)}_0(\rho) = \mathrm{ord}^{(10)}_0(x) -1$.  By the Weierstrass gap theorem \cite{Paule3},

\begin{align*}
\mathcal{M}^{0}\left( \Gamma_0(10) \right) = \mathbb{C}[x],
\end{align*} ensuring that $\rho\in\mathbb{C}[x]$ and $t\in\mathbb{C}[x^{-1},x]$.

To more easily examine the expansion of $\rho, xt$ at 0, we can first apply the transformation $\tau\rightarrow -1/\tau$ and examine their expansions (in a uniformitized variable) at $\infty$:

\begin{align*}
x\left( -1/\tau \right) = \frac{1}{4} \frac{\eta(\tau/2)^5\eta(\tau/5)}{\eta(\tau)^5\eta(\tau/10)} = \frac{1}{4}h(\tau/10).
\end{align*}  It is easy to verify that

\begin{align*}
h(\tau) = \frac{\eta(2\tau)\eta(5\tau)^5}{\eta(\tau)\eta(10\tau)^5}\in\mathcal{M}^{\infty}\left( \Gamma_0(10) \right).
\end{align*}  Similarly, we have

\begin{align*}
t\left( -1/\tau \right) &= \frac{1}{25} \frac{\eta(\tau/5)^2\eta(\tau/10)^2}{\eta(\tau)^2\eta(\tau/2)^2} = \frac{1}{25}t(\tau/10)^{-1},\\
\rho\left( -1/\tau \right) &= \frac{1}{5} \frac{\eta(\tau/2)^2\eta(\tau/5)^4}{\eta(\tau)^4\eta(\tau/10)^2} = \frac{1}{5}\left(4\cdot x(-1/\tau)+1\right),
\end{align*} and
\begin{align*}
x\left( -1/\tau \right)\cdot t\left( -1/\tau \right) &= \frac{1}{100} h(\tau/10)\cdot t(\tau/10)^{-1}\\
&= -\frac{1}{25}-\frac{3}{100}h(\tau/10)+\frac{1}{100}h(\tau/10)^2\\
&= -\frac{1}{25}-\frac{3}{25}x(-1/\tau)+\frac{4}{25}x(-1/\tau)^2.
\end{align*} We can now map $-1/\tau$ back to $\tau$, and then isolate $\rho$ and $t$:

\begin{align*}
\rho\left( \tau \right) &= \frac{1}{5}\left(4\cdot x(\tau)+1\right),\\
t\left(\tau \right) &= \frac{1}{25} \left( -x(\tau)^{-1}-3+4\cdot x(\tau)\right).
\end{align*}  Our first attempt to describe $L_1$ in terms of a single function arose from these substitutions into (\ref{L1WY}).

Notice that $h$ is a hauptmodul, i.e., $\mathcal{M}^{\infty}\left( \Gamma_0(10) \right) = \mathbb{C}[h]$.  In the case of the Andrews--Sellers congruences, the corresponding functions are modular over $\mathrm{X}_0(20)$, which has genus 1.  By the Weierstrass gap theorem, we cannot reduce $\mathcal{M}^{\infty}\left( \Gamma_0(20) \right)$ to a 1-variable polynomial ring; indeed, $\mathcal{M}^{\infty}\left( \Gamma_0(20) \right)$ must be isomorphic to a rank 2 module $\left< 1, h_3 \right>_{\mathbb{C}[h_2]}$, for $\mathrm{ord}^{(10)}_{\infty}(h_2) = -2$, $\mathrm{ord}^{(10)}_{\infty}(h_3) = -3$.

Because of this, we are forced to take a rank 2 module, and therefore a two-variable system.  From this, the necessity of the Paule--Radu apparatus must follow.

As a final important note on the potential of our technique to produce new results, we give a very brief description of how we discovered that a localized ring structure was more useful than the standard polynomial ring.  Our initial attempt to describe $L_{\alpha}$ in terms of the function $x$ failed almost immediately: one can verify that with the appropriate substitutions of the function $x$ into (\ref{L1WY}), we get

\begin{align}
\frac{L_1}{F} =& - \frac{624}{625 x^3} - \frac{2487}{625 x^2} + \frac{801}{625 x} -\frac{422}{125}  - \frac{3148 x}{125}\nonumber\\ &+ \frac{19904 x^2}{625} + \frac{512 x^3}{625} - \frac{256 x^4}{625}.\label{L1x}
\end{align}  This clearly does not work.

However, at the advice of Silviu Radu, we attempted to adjust our function $x$.  We discovered the appropriate substitution in the form of $x=1+5y$.  Substituting into (\ref{L1x}) and simplifying, we derive (\ref{L1S}).

The critical point is that we would prefer the function needed to annihilate the poles of $L_{\alpha}$ to be equal to (or a power of) the function used to describe the right-hand side of the witness identity.  This is the situation that the author has studied, together with Radu, in \cite{Raduns}.  It is of course far more probable that these functions are not equal.  Nevertheless, if the function used on the right-hand side, e.g., $y$, is a hauptmodul, then we could still use this function to describe our prefactor function (e.g., $x=1+5y$).  This necessarily induces a localized ring.

Such ring structure may appear more complicated, and more daunting, than that of the traditional methods.  Nevertheless, as we hope to have shown, the resulting complications and fears can be overcome.

\pagebreak
\section{Appendix I}

We provide the tables used in the proof of Theorem \ref{w12w1}.  These can easily be constructed by hand.  We provide additional details on this and other computations in our Mathematica supplement at \url{https://www3.risc.jku.at/people/nsmoot/online3.nb}.

\begin{table}[hbt!]
\begin{center}
\begin{tabular}{l|c|c|r}
 $m$      &$r=1$  & $r=2$ & $r=3$ \\
\hline
 1         & -1            & 0 & 1 \\
 2         & -1            & 0 & 1  \\
 3         & -1            & 1 & 0   \\
 4         & 0             & 1 & 1   \\
 5         & 1             & 2 & 1  \\
 6         &              & 2 & 2  
 \end{tabular}
\caption{Value of $\displaystyle{\theta(m) + \pi_1(m,r) + \pi_0(r,1) -2}$ with $1\le m\le 6,$  $1\le r\le 3$}\label{tablew1}
\end{center}
\end{table}

\begin{table}[hbt!]
\begin{center}
\begin{tabular}{l|c|c|r}
 $m$      &$r=1$     & $r=2$ & $r=3$ \\
\hline
 1         & -1            & 0     & 1 \\
 2         & -1            & 0     & 1 \\
 3         & -1            & 1     & 0 \\
 4         & 0             & 1      & 1\\
 5         & 1             & 2      & 1\\
 6         &              & 2       & 2 
 \end{tabular}
\caption{Value of $\displaystyle{\theta(m) + \pi_1(m,r) + \pi_0(r,2) -2}$ with $1\le m\le 6,$  $1\le r\le 3$}\label{tablew2}
\end{center}
\end{table}

\begin{table}[hbt!]
\begin{center}
\begin{tabular}{l|c|c|r}
 $m$      &$r=1$     & $r=2$ & $r=3$ \\
\hline
 1         & 0            & 0     & 2\\
 2         & 0            & 0     & 2\\
 3         & 0            & 1     & 1\\
 4         & 1             & 1      & 2\\
 5         & 2             & 2      & 2\\
 6         &              & 2       & 3  
\end{tabular}
\caption{Value of $\displaystyle{\theta(m) + \pi_1(m,r) + \pi_0(r,3) -2}$ with $1\le m\le 6,$  $1\le r\le 3$}\label{tablew3}
\end{center}
\end{table}

\pagebreak

\section{Appendix II}

Below we list the ten fundamental relations that are justified using our cusp analysis in Section 6.  For the complete derivation of the 50 relations used in Theorem \ref{thmuyox}, see our Mathematica supplement online at \url{https://www3.risc.jku.at/people/nsmoot/online3.nb}.

\begin{align}
\text{Group I:}&\nonumber\\
U^{(1)}\left( 1 \right) &=\frac{1}{1+5y}\left( 1+5^2y+16\cdot 5\cdot y^2 \right)\\
U^{(1)}\left( y \right) &= y\\
U^{(1)}\left( y^2 \right) &=51 y + 471\cdot 5\cdot y^2 + 1364\cdot 5^2\cdot y^3 + 1776\cdot 5^3\cdot y^4\nonumber\\ &+ 1088\cdot 5^4\cdot y^5 + 256\cdot 5^5\cdot y^6\\
U^{(1)}\left( y^3 \right) &= 41 y + 2474\cdot 5\cdot y^2 + 29193\cdot 5^2\cdot y^3 + 152248\cdot 5^3\cdot y^4\nonumber\\ &+ 2231024\cdot 5^3\cdot y^5 +  814336\cdot 5^5\cdot y^6 + 4833536\cdot 5^5\cdot y^7\nonumber\\ &+ 3753984\cdot 5^6\cdot y^8 +  1847296\cdot 5^7\cdot y^9 + 524288\cdot 5^8\cdot y^{10}\nonumber\\ &+ 65536\cdot 5^9\cdot y^{11}\\
U^{(1)}\left( y^4 \right) &=11 y + 3981\cdot 5\cdot y^2 + 138181\cdot 5^2\cdot y^3 + 8956203\cdot 5^2\cdot y^4\nonumber\\ &+ 62033852\cdot 5^3\cdot y^5 +  53739872\cdot 5^5\cdot y^6 + 791357952\cdot 5^5\cdot y^7\nonumber\\ &+ 1662808832\cdot 5^6\cdot y^8 +  2561985536\cdot 5^7\cdot y^9\nonumber\\ &+ 14663327744\cdot 5^7\cdot y^{10} +  2496888832\cdot 5^9\cdot y^{11}\nonumber\\ &+ 7817854976\cdot 5^9\cdot y^{12} +   3503816704\cdot 5^{10}\cdot y^{13}\nonumber\\ &+ 1065353216\cdot 5^{11}\cdot y^{14} +  197132288\cdot 5^{12}\cdot y^{15}\nonumber\\ &+ 16777216\cdot 5^{13} y^{16}\\
\text{Group II:}&\nonumber\\
U^{(0)}\left( 1 \right) &=\frac{1}{1+5y}\left( -5y-4\cdot 5\cdot y^2 \right)\\
U^{(0)}\left( y \right) &= 5y+4\cdot 5\cdot y^2\\
U^{(0)}\left( y^2 \right) &=5 y + 153\cdot 5\cdot y^2 + 3956\cdot 5\cdot y^3 + 8528\cdot 5^2\cdot y^4 + 9152\cdot 5^3\cdot y^5\nonumber\\ &+ 4864\cdot 5^4\cdot y^6 +  1024\cdot 5^5\cdot y^7\\
U^{(0)}\left( y^3 \right) &= y + 1874 y^2 + 40101\cdot 5\cdot y^3 + 309864\cdot 5^2\cdot y^4\nonumber\\ &+ 1252624\cdot 5^3\cdot y^5 + 
 3071232\cdot 5^4\cdot y^6 + 4892928\cdot 5^5\cdot y^7\nonumber\\ &+ 26039296\cdot 5^5\cdot y^8 +  18464768\cdot 5^6\cdot y^9 + 8404992\cdot 5^7\cdot y^{10}\nonumber\\ &+ 2228224\cdot 5^8\cdot y^{11} +  262144\cdot 5^9\cdot y^{12}\\
U^{(0)}\left( y^4 \right) &= 329\cdot 5\cdot y^2 + 116926\cdot 5\cdot y^3 + 2285653\cdot 5^2\cdot y^4\nonumber\\ &+ 21410212\cdot 5^3\cdot y^5 +  119101984\cdot 5^4\cdot y^6 + 438497152\cdot 5^5\cdot y^7\nonumber\\ &+ 45458688\cdot 5^8\cdot y^8 +  2150618112\cdot 5^7\cdot y^9 + 3033554944\cdot 5^8\cdot y^{10}\nonumber\\ &+  3217784832\cdot 5^9\cdot y^{11} + 12811829248\cdot 5^9\cdot y^{12}\nonumber\\ &+  37793038336\cdot 5^9\cdot y^{13} + 16051601408\cdot 5^{10}\cdot y^{14}\nonumber\\ &+  4647288832\cdot 5^{11}\cdot y^{15} + 822083584\cdot 5^{12}\cdot y^{16}\nonumber\\ &+  67108864\cdot 5^{13}\cdot y^{17}
\end{align}

\pagebreak

\section{Acknowledgments}
The research was funded by the Austrian Science Fund (FWF): W1214-N15, project DK6, and by the strategic program ``Innovatives OÖ 2010 plus" by the Upper Austrian Government.  My fondest thanks to the Austrian Government for their generous support.

I am extremely grateful to Liuquan Wang and Yifan Yang for their work on this family of congruences, without which this project would never have developed.  I am enormously grateful to Drs. Cristian-Silviu Radu and Johannes Middeke for their advice on technical matters, and to Koustav Banerjee for calming my doubts and insisting that my insights were worth investigating.  Finally, an enormous thanks to Professor Peter Paule for his patience, as the prospects of this project waxed and waned time and again.

\end{document}